\documentclass{article}
\usepackage{amssymb, amscd,  mathrsfs}
\usepackage{latexsym}
\usepackage{amsmath}
\usepackage{amsthm}
\usepackage{amsfonts, float}
\usepackage{lineno, hyperref}
\usepackage{adjustbox}
\usepackage{subfigure}
\usepackage{indentfirst}
\usepackage{enumerate,longtable, tabu}
\usepackage{mathtools}
\PassOptionsToPackage{unicode}{hyperref}
\DeclarePairedDelimiter{\diagfences}{(}{)}
\newcommand{\diag}{\operatorname{diag}\diagfences}

\graphicspath{{./img/}{./figures}}
\usepackage{tikz}

\usetikzlibrary{calc}

\usepackage[figurename=Fig.]{caption}

\modulolinenumbers[5]

\newtheorem{thm}{Theorem}[section]
\newtheorem{lem}{Lemma}[section]
\newtheorem{definition}{Definition}[section]

\newtheorem{proof*}{proof}[section]
\newtheorem{cor}{Corollary}[section]
\newtheorem{con}{Conjecture}[section]

\newtheorem{problem}{Problem}

\newtheorem{remark}{Remark}
\newtheorem{fact}{Fact}

\usepackage{bbding}
\newcommand{\Rose}{\vcenter{\hbox{\SixFlowerPetalDotted}}}

\DeclareRobustCommand\Inf{%
	\mathop{
		\kern-2pt\vcenter{\hbox{
				\tikz[anchor=base,  baseline=-3pt, line width = 0.6pt, xscale=0.8, yscale=0.8]{
					\draw (-0.4em, 0) circle (0.2em and 0.6em);
					\draw (0em, 0) circle (0.2em and 0.6em);
					\draw (0.4em, 0) circle (0.2em and 0.6em);
				}
		}}\kern-2 pt}
}

\DeclareRobustCommand\INF{%
	\kern-2pt\vcenter{\hbox{
			\tikz[anchor=base,  baseline=-3pt, line width = 0.4pt,  xscale=0.8, yscale=0.8]{
				\draw[double,  double distance=0.2pt] (-0.4em, 0) circle (0.2em and 0.6em);
				\draw[double,  double distance=0.2pt] (0em, 0) circle (0.2em and 0.6em);
				\draw[double,  double distance=0.2pt] (0.4em, 0) circle (0.2em and 0.6em);
			}
	}}\kern-2 pt
}

\title{Some $\alpha$-spectral extremal results for some digraphs}
\author{
Haiying Shan\thanks{\footnotesize School of Mathematical Sciences, Tongji University, Shanghai 200092, China
(\texttt{shan\_haiying@tongji.edu.cn})},~~
Feifei Wang\thanks{School of Mathematical Sciences, Tongji University, Shanghai 200092, China
(\texttt{1710854@tongji.edu.cn})},
Changxiang He\thanks{College of Science, University of Shanghai for Science
and Technology, Shanghai,
P. R. China
(\texttt{changxiang-he@163.com})}~,~~
}

\begin{document}
\maketitle

\begin{abstract}
			
	In this paper,  we characterize the extremal digraphs with  the maximal or minimal $\alpha$-spectral radius among some digraph classes such as rose digraphs, generalized theta digraphs and tri-ring digraphs with given size $m$. These digraph classes are denoted by $\mathcal{R}_{m}^k$,  $\widetilde{\boldsymbol{\Theta}}_k(m)$ and $\INF(m)$  respectively. The main results about spectral extremal digraph by Guo and Liu in \cite{MR2954483} and Li and Wang in \cite{MR3777498} are generalized to $\alpha$-spectral graph theory.  As a by-product of our main results, an open problem in \cite{MR3777498} is answered.			
	Furthermore,  we determine the digraphs with  the first three minimal $\alpha$-spectral radius among all strongly connected digraphs. Meanwhile,  we determine the unique digraph with  the fourth minimal $\alpha$-spectral radius among all strongly connected digraphs for $0\le \alpha \le \frac{1}{2}$.

    \bigskip


\end{abstract}

	\section{Introduction}
	Let $G$ be a digraph of order $n$ with vertex set $V(G)=\{v_1, \dots, v_n\}$ and arc set $E(G)\subset V(G)\times V(G)$. Throughout this paper,  we assume that $G$ is finite and simple,  i.e.,  without loops and multiple arcs. 
	For a digraph $D$,  if two vertices are connected by an arc,  then they are called adjacent. If there is an arc from $v_i$ to $v_j$,  we indicate this by writing $(v_i, v_j)$,  call $v_j$ the head of $(v_i, v_j)$,  and $v_i$ the tail of $(v_i, v_j)$,  respectively. The digraph $G$ is strongly connected if for every pair of vertices $v_i, v_j\in V(G)$,  there exists a directed path from $v_i$ to $v_j$ and a directed path from $v_j$ to $v_i$. For any vertex $v_i$,  let $N_i^+(G)=\{v_j|(v_i, v_j)\in E(G)\}$ and $N_i^-(G)=\{v_j|(v_j, v_i)\in E(G)\}$ denote the out-neighbors and in-neighbors of $v_i$,  respectively. Let $d_i^+(G)=|N_i^+(G)|$
	denote the outdegree of the vertex $v_i$,  and $d_i^-(G)=|N_i^-(G)|$ denote the indegree of the vertex $v_i$ in the digraph $G$.  Sometimes we simply write $d_v^{+}(G)$ and  $d_v^{-}(G)$ as $d_v^{+}$  and $d_v^{-}$, respectively, if there exists no confusion.

	Spectral graph theory studies connections between combinatorial properties of graphs and the eigenvalues of matrices associated to  graphs,  such as the adjacency matrix, the signless Laplacian matrix etc. In 2017,  Nikiforov \cite{MR3648656} first introduced the concept of the $ A_\alpha$ matrix of an ordinary graph. 
	
	The study of $A_\alpha$ matrix has started attracting attention of researchers in recent years(see \cite{MR3786248, MR3771882, MR3876180, MR3842581,  guo2018alpha1}). In \cite{MR3872980}, Liu et al.  introduced the concept of the $A_\alpha$ matrix for digraph as follows:  for $0 \leq \alpha <1$, 
	$$A_{\alpha}(G)=\alpha D(G)+(1-\alpha)A(G),$$ 
	where $A(G)$ is the adjacency matrix of $G$ and $D(G)=\diag{d_1^+, d_2^+, \dots, d_n^+}$ is the outdegree diagonal matrix of $G$. For square matrix $A$, let $\phi(A,x)=\det (xI-A)$ denote the characteristic polynomial of $A$ and $\rho(A)$ be the spectral radius of $A$.  We write $\phi_{\alpha}(G)=\phi(A_{\alpha}(G), x)$ and $\rho_\alpha(G)=\rho(A_{\alpha}(G))$.  For other standard notations and terminologies of digraphs and spectral graph theory, readers are referred to  \cite{jorgenbangjensenDTA2009} and \cite{brouwerSpectraGraphs2012}.

	In \cite{MR3872980},  Liu et al. characterized the extremal digraph which attains the maximal $\alpha$-spectral radius among all strongly connected digraphs with given dichromatic number and characterized the extremal digraphs with  the maximal (resp. minimal) $\alpha$-spectral radius among all strongly connected bicyclic digraphs. In \cite{Xi2018Merging},  Xi and So determined the digraphs which attains the maximal (or minimal) $\alpha$-spectral radius among all strongly connected digraphs with given parameters such as girth,  clique number,  vertex connectivity or arc connectivity. In \cite{MR3230435},  Hong and You determined the unique digraph with  the minimal (or maximal),  the second minimal (or maximal),  the third minimal,  the fourth minimal adjacency spectral radius and signless Laplacian spectral radius among all strongly connected digraphs. 
	In \cite{MR2954483},  Guo and Liu characterized the extremal digraphs with  the maximal and minimal adjacency spectral radius among two kinds of generalized $\infty$ and $\Theta$-digraphs. In \cite{MR3777498},  Li and Wang obtained some results about the signless Laplacian spectral radius of generalized $\infty$ and $\Theta$-digraphs. The digraph with  the maximum signless Laplacian spectral radius is determined among all $\widetilde{\infty}_2(p, q, s)$-digraphs of order $n$ with $q\le p\le s$. The following problem was proposed in \cite{MR2954483}:
	
	\begin{problem}\label{prob1}
		Among all $\widetilde{\infty}_2(p, q, s)$-digraphs of order $n(=p+q+s+2)$ with $1\le p\le s$, which digraph achieves the maximal (or minimal) signless Laplacian spectral radius?
	\end{problem}
	
	In this paper, motivated by the above results, we will  solve the above question
	for $\alpha$-spectral radius and determine the extremal digraphs with maximal or minimal  $\alpha$-spectral radius among some classes of digraphs such as $\mathcal{R}_{m}^k$, $\widetilde{\boldsymbol{\Theta}}_k(m)$, $\displaystyle \INF(m)$ and so on. These notations of classes of digraphs are introduced in Section 3 and Section 4 of this paper and those digraphs are considered as extensions of definitions of  $\infty$ and $\Theta$-digraphs and the digraph $\widetilde{\infty}_2(p, q, s)$ is denoted by $\displaystyle\Inf(p+1, q+2, s+1)$. As a by-product of our main results (Theorem \ref{thmd4}), Problem \ref{prob1} is answered.

	The remaining of this paper is organized as follows. In Section 2,  some preliminary results which will be needed to
	prove our main results are presented. 
	In Section $3$,  we characterize the extremal digraphs which attain the maximal and minimal $\alpha$-spectral radius among $\mathcal{R}_{m}^k$. We also study the extremal $\alpha$-spectral radius question for digraphs among $\displaystyle \INF(m)$. 
	In Section $4$,   we characterize the extremal digraphs which attain the maximal and minimal $\alpha$-spectral radius among $\widetilde{\boldsymbol{\Theta}}_k^1(m)$. We also obtain the maximal $\alpha$-spectral radius among $\prescript{s}{t}{\widetilde{\boldsymbol\Theta}(m)}$. 
	Furthermore,  we determine the extremal digraphs which attain the maximal $\alpha$-spectral radius among $\mathcal{R}_{m}^k$ and $\widetilde{\boldsymbol{\Theta}}_k(m)$ and the minimal $\alpha$-spectral radius among $\mathcal{R}_{m}^k$ and $\widetilde{\boldsymbol{\Theta}}_k^1(m)$. In Section 5,  we determine the digraphs with  the first three minimal $\alpha$-spectral radius among all strongly connected digraphs. Furthermore,  we determine the unique digraph with  the fourth minimal $\alpha$-spectral radius among all strongly connected digraphs for $0\le \alpha \le \frac{1}{2}$.

	\section{Preliminary results}  
	
	In this section, we present some preliminary results which will be
	needed to prove our main results.

	In the following lemma, we will list some results on nonnegative matrices and digraphs, which  follow directly from
	Perron-Frobenius theorem (see \cite{bermanNMM2014}).
	\begin{lem}\label{lemc2} Let $G$ is a strongly connected digraph. Some basic results on the spectrum are the following:
		\begin{enumerate}[(i).]
			\item  $\rho_\alpha(G)$ is the maximum eigenvalue of $A_\alpha(G)$, with a unique normalized positive eigenvector $\boldsymbol{x}$, which is called the $\alpha$-Perron vector of $G$.
			\item   Let $A$ and $B$ be two nonnegative matrices with $A \geq B$ and $A \neq B$. If  there exists a principle submatrix $M$ of $A$ such that $B \leq M$, then $\rho(B)\leq \rho(A)$ and the inequality is strict when $A$ irreducible.
			\item $\delta^+(G)\le \rho_{\alpha}(G)\le \Delta^+(G)$ and  either one equality holds if and only if $G$ has a constant outdegree vector.
			\item If $\alpha \boldsymbol{x}\le A_\alpha(G)\boldsymbol{x}\le \beta \boldsymbol{x}$, then   $\alpha \le\rho_\alpha(G) \le \beta$.
		\end{enumerate}
	\end{lem}

	\begin{lem} \label{lemz2}  
		Let $A$ and $B$ be two nonnegative matrices of order $n$ with $A \geq B$ and $A \neq B$. Then the following holds:
		\begin{center}
			$\phi(B,x) \geq \phi(A,x)$ for $x \geq \rho(A)$,
		\end{center}
		especially, when $A$ is irreducible, the inequality is strict.
	\end{lem}
	\begin{proof}The proof is by induction on $n$, the order of $A$.
		The case $n = 1$ is obvious.
		Assume that the result holds when $n \leq k$.  Now consider the case for $n=k+1$.
		It is well known that the derivative of characteristic polynomial of $A$ is the sum of characteristic
		polynomials of the $n$ principal submatrices of $A$ of size $n-1$, which we denote by $A_1,\cdots, A_n$, that is,
		$$\phi'(A,x)=\sum_{i=1}^{n}\phi(A_i,x).$$
		Hence,
		\begin{equation} \label{chp:def}
		\phi'(B,x)-\phi'(A,x)=\sum_{i=1}^{n}(\phi(B_i,x)-\phi(A_i,x)).
		\end{equation}
		Since $A \geq B, A\neq B$, $A_i \geq B_i$ for  $i=1,\dots,n$. From (ii) of Lemma \ref{lemc2},    we have $\rho(A) \geq \rho(B)$ and $\rho(A)>\rho(A_i) \geq \rho(B_i)$.  By the inductive hypothesis,
		\begin{center}
			$\phi(B_i,x) \geq \phi(A_i,x)$ hold when $x \geq  \rho(A) \geq \rho(A_i)$.
		\end{center}
		Thus, by Formula \eqref{chp:def}, we have
		\begin{center}
			$\phi'(B,x)-\phi'(A,x)\geq 0$ for $x \geq  \rho(A)$,
		\end{center}
		which implies that the function $g(x)=:\phi(B,x)-\phi(A,x)$ is monotonically increasing on the interval $[\rho(A), +\infty)$.
		
		So, $g(x) \geq g(\rho(A))=\phi(B,\rho(A))-\phi(A,\rho(A)) \geq 0$ for $x \geq \rho(A)$,
		when $A$ is irreducible, $\rho(A)>\rho(B)$ and $g(\rho(A))>0$, which implies that
		$\phi(B,x) > \phi(A,x)$ for $x \geq \rho(A)$.  
		
		Therefore, the result holds when $n=k+1$.  By induction principle,
		the desired result follows.
	\end{proof}

	Let $U$ be a vertex subset of digraph $G$ and let $A_\alpha(G)(U)$ be the principal submatrix of $A_\alpha(G)$ obtained by deleting the rows and columns corresponding to vertices of $U$.  We write $\psi_{\alpha}(G,U)=\phi(A_\alpha(G)(U),x)$.   When $U=\{v\}$,  we rewrite  $A_\alpha(G)(U)$ as $(A_\alpha)_v(G)$ and $\psi_{\alpha}(G,v)=\phi(A_\alpha(G)(U),x)$.
	
	Suppose that $G_1,  G_2, \dots, G_k$ are $k$ disjoint connected digraphs and $v_i \in V(G_i)$ for $i=1, 2, \cdots, k$.
	The coalescence of $G_1,  G_2, \dots, G_k$ with respect to $v_1, v_2,\dots,v_k$, denoted by $\bigodot\limits_{i=1}^kG_i(v_i)$,   is the digraph  obtained  from $G_1,  G_2, \dots$, $G_k$ by identifying vertices $v_1, v_2, \ldots, v_k$ as new vertex $v$. If $k=2$, the  coalescence of $G_1$ and $G_2$ with respect to $v_1, v_2$ also is written as $G_1v_1:v_2G_2$ (or $G_1\cdot G_2$ for short).

	\begin{thm}\label{thm:b1}
		Let $H=\bigodot\limits_{i=1}^kG_i(v_i)$.  Then we have
		\begin{align}
		\phi_{\alpha}(H)=\prod_{i=1}^k\psi_\alpha(G_i,v_i)\left((1-k)x+\sum_{i=1}^k\frac{\phi_{\alpha}(G_i)}{\psi_\alpha(G_i,v_i)}\right).
		\end{align}
	\end{thm}
	
	\begin{proof}
		For short, set $A_i=(A_\alpha)_{v_i}(G_i)$ for $i=1, 2, \dots, k$.  Then $\phi(A_i)=\psi_\alpha(G_i,v_i)=\psi_\alpha(G_i,v)$.
		Without loss of generality,  we may assume the $\alpha$-matrix of $H$ is in the following form:
		\begin{align*}
		A_{\alpha}(H)=\begin{pmatrix}
		\alpha d       & \beta_1 & \beta_2 & \cdots & \beta_{k-1} & \beta_k \\
		\gamma_1      & A_1     & 0\,     & \cdots & 0           & 0       \\
		\gamma_2      & 0       & A_2     & \cdots & 0\,         & 0       \\
		\vdots         & \vdots  & \vdots  & \cdots & \vdots      & \vdots  \\
		\gamma_{k-1} & 0       & 0       & \cdots & A_{k- 1}    & 0       \\
		\gamma_{k}   & 0       & 0       & \cdots & 0           & A_{k}
		\end{pmatrix},
		\end{align*}
		where the first row of $A_\alpha(H)$ is corresponding to vertex $v$ of $H$, 
		$d=\sum\limits_{i=1}^kd_{v_i}^+$  and $\beta_i, \gamma_i$ $(i=1, 2, \dots, k)$ are vectors 
		of suitable dimensions.
		
		By the Laplace expansion formula for the determinant,  we have
		\begin{align*}
		\phi_{\alpha}(H) & =\prod_{i=1}^k\phi(A_{i})\left [\sum_{i=1}^k\frac{\phi_{\alpha}(G_{i})}{\phi(A_{i})}-(k-1)x\right ]\\
		&=\prod_{i=1}^k\psi_\alpha(G_i,v_i)((1-k)x+\sum_{i=1}^k\frac{\phi_{\alpha}(G_i)}{\psi_\alpha(G_i,v_i)}).  \qedhere
		\end{align*}
	\end{proof}

	The following corollaries can be derived from Theorem \ref{thm:b1}.
	
	\begin{cor}  \label{cord2}
		Let $G$ and $H$ be two digraphs with $u\in V(G)$ and $v\in V(H)$. Then
		\begin{equation}\label{coal:alpha}
		\phi_{\alpha}(Gu:vH)= \phi_{\alpha}(G)\psi_\alpha(H,v)+\psi_\alpha(G,u)\phi_{\alpha}(H) - x\psi_\alpha(G,u)\psi_\alpha(H,v).
		\end{equation}
	\end{cor}

	Let $G_v(n_1, n_2, \dots, n_k)$ be the digraph obtained from strongly connected digraph $G$ by attaching $k(k\ge 2)$ directed cycles of length $n_i$ $(i=1, \dots, k)$ at $v$,  respectively. In the sequel,  unless otherwise
	stated, we will denote $d=\frac{1-\alpha}{x-\alpha}$.
	\begin{cor}\label{cor:b1}
		Let $H=G_v(n_1, n_2, \dots, n_k)$ and $n=\sum\limits_{i=1}^k(n_i-1)$,  $d=\frac{1-\alpha}{x-\alpha}$  and  $A'=(A_{\alpha})_{v}(G)$,  then
		\begin{align*}
		\phi_{\alpha}(H)=(x-\alpha)^n\left(\phi_{\alpha}(G)-\phi(A')(k\alpha+(1-\alpha)\sum_{i=1}^kd^{n_i - 1})\right).
		\end{align*}
	\end{cor}
	
	\begin{proof} Let $F_i(x)=\phi(A_{\alpha}(\overrightarrow{C}_{n_i})), f_i(x)=\psi_\alpha(\overrightarrow{C}_{n_i},v)$ for $i=1, 2, \dots, k$.
		By direct calculation, we have
		\begin{center}
		 	$F_i(x)=(x-\alpha)^{n_i } -(1-\alpha)^{n_i } $
			and  $f_i(x)=(x-\alpha)^{n_i - 1}$.
		\end{center}
		From Theorem \ref{thm:b1},  we have
		\begin{align*}
		\phi_{\alpha}(H) & =\phi(A')\prod_{i=1}^{k}f_i(x)
		\left(-kx+\frac{\phi_{\alpha}(G)}{\phi(A')}+ \sum_{i=1}^{k}\frac{F_i(x)}{f_i(x)}\right)                                                  \\
		& =(x-\alpha)^n\phi(A')\left(-kx+k(x-\alpha)-(1-\alpha)\sum_{i=1}^kd^{n_i - 1}+\frac{\phi_{\alpha}(G)}{\phi(A')}\right) \\
		& =(x-\alpha)^n\left(\phi_{\alpha}(G)-\phi(A')(k\alpha+(1-\alpha)\sum_{i=1}^kd^{n_i - 1})\right).  \qedhere
		\end{align*}
	\end{proof}
	
	Let $A= \begin{pmatrix}
	A_1 & \alpha_1\\\beta_1 & a
	\end{pmatrix}, B=\begin{pmatrix}
	b & \beta_2\\\alpha_2 & B_1
	\end{pmatrix}$ and $C=\begin{pmatrix}
	A_1 & \alpha_1 & O\\
	\beta_1 & a+b & \beta_2\\
	O & \alpha_2 & B_1
	\end{pmatrix}$ be square matrices with $a,b \in \mathbb{R}$. By similar argument of proof Theorem \ref{thm:b1}, the following formula holds. 
	\begin{equation}\label{coal:mat}
	\phi(C,x)= \phi(A,x)\phi(B_1,x)+\phi(A_1,x)\phi(B,x) - x\phi(A_1,x)\phi(B_1,x).
	\end{equation}
	
	Let $G_{u, v}(s, t)$ be the digraph obtained from digraph $G$ by attaching two directed cycles of length $s$ and $t$ at $u$ and $v$,  respectively.
	Applying  Formulae \eqref{coal:alpha} and \eqref{coal:mat},  the following result can be deduced.
	\begin{lem}\label{thmd1}
		Let $G$ be a digraph,  vertex $u, v\in V(G)$ and $u\neq v$.
		Let $M=A_{\alpha}(G)$,  $M_1=(A_{\alpha})_{u}(G)$,  $M_2=(A_{\alpha})_{v}(G)$,  $M_3=A_{\alpha}(G)(\{u,v\})$ and $f_i=\phi(M_i, x)$,  $(i=0, 1, 2, 3), \ M_0=M$. We have
		\begin{equation}\label{eq:b1}
		\begin{aligned}
		\phi_{\alpha}(G_{u, v}(s, t)) & = (x-\alpha)^{s+t-2}\big( \alpha^2f_3-\alpha(f_1+f_2)+f_0\big)     \\
		& +(x-\alpha)^{t-1}(1-\alpha)^s(\alpha f_3-f_1)                \\
		& + (x-\alpha)^{s-1}(1-\alpha)^t(\alpha f_3-f_2)+(1-\alpha)^{s+t}f_3.
		\end{aligned}		
		\end{equation}
	\end{lem}

	In \cite{belardoCombinatorialApproachComputing2010a}, Francesco Belardo et al. present a Schwenk-like formula for weighted digraphs.
	\begin{thm}[\cite{belardoCombinatorialApproachComputing2010a}]
		Let $A=(a_{i,j})$ be any square matrix of order $n$, and let $G\left(=G_{A}\right)$ be its Coates digraph with $V(G)=\{1,2,\dots,n\}$. If $v$ is a fixed vertex of G then
		$$
		\phi(G)=\left(x-a_{v v}\right) \phi(G-v)-\sum_{\overrightarrow{C} \in \mathcal{C}_{v}(G)} \omega_{A}(\overrightarrow{C}) \phi(G-V(\overrightarrow{C})),
		$$
		where $\mathcal{C}_{v}(G)$ is the set of directed cycles of $G$ of length $\geqslant 2$ passing through $v$ and $\omega_{A}(\overrightarrow{C})=\prod \limits_{(i,j) \in E(\overrightarrow{C})} a_{i j}$.
	\end{thm}
	
	Applying above theorem to the $\alpha$-matrix of digraph $G$,  the following lemma holds.
	\begin{lem}\label{lemd3}
		For any vertex $v$ of the digraph $G$,  let $v$ be a vertex of $G$. The characteristic polynomial of $G$ satisfies
		$$\phi_{\alpha}(G)=(x-\alpha d_v^+) \psi_\alpha(G,v)-\sum_{\overrightarrow{C} \in \mathcal{C}_{v}(G)}(1-\alpha)^{|V(\overrightarrow{C} )|}\psi_\alpha(G,V(\overrightarrow{C} )),$$
		where $\mathcal{C}_{v}(G)$ is the set of directed cycles of $G$ of length $\geqslant 2$ passing through $v$.
	\end{lem}
	
	The following lemma is a key ingredient of the proof of main results.
	\begin{lem}\label{lemd4}
		Suppose that $u, v$ are two internal vertices of internal directed path of digraph $G$ and $w$ is a vertex of digraph $H$. Then,  for $0 \leq \alpha \leq 1$,   \[
		\phi_{\alpha}(Gu:wH) =\phi_{\alpha}(Gv:wH).
		\]
	\end{lem}
	\begin{proof}
		By Corollary \ref{cord2},  we have
		\begin{align*}
		\phi_{\alpha}(Gu:wH) & =\psi_\alpha(G,u)(\phi_{\alpha}(H)-x\psi_\alpha(H,w))+\phi_{\alpha}(G)\psi_\alpha(H,w), \\
		\phi_{\alpha}(Gv:wH) & =\psi_\alpha(G,v)(\phi_{\alpha}(H)-x\psi_\alpha(H,w))+\phi_{\alpha}(G)\psi_\alpha(H,w).
		\end{align*}
		
		Since $d^+_u(G)=d^+_v(G)$ and $\mathcal{C}_{u}(G)=\mathcal{C}_{v}(G)$, 
		from Lemma \ref{lemd3},  we have
		\begin{align*}
		\psi_\alpha(G,u) =\psi_\alpha(G,v).
		\end{align*}
		Hence, 
		\begin{align*}
		& \phi_{\alpha}(Gu:wH)-\phi_{\alpha}(Gv:wH)\,                                                         \\
		= & (\psi_\alpha(G,u)-\psi_\alpha(G,v))(\phi_{\alpha}(H)-x\psi_\alpha(H,w))=0. \qedhere
		\end{align*}
	\end{proof}

	For the real vector $\mathbf{x}=(x_1,..., x_p)$, $p\in \mathbb{N}$,
	denote with $\mathbf{x}_\uparrow=(x_{[1]},..., x_{[p]})$ the vector
	where all components of $\mathbf{x}$ are arranged in ascending
	order.

	\begin{definition} \textbf{\cite{marshallInequalitiesTheoryMajorization2011}}
		A non-negative vector $\mathbf{x} = (x_1, ..., x_p)$, $p \in
		\mathbb{N}$, \emph{weakly majorizes} a non-negative vector
		$\mathbf{y} = (y_1, ..., y_p)$ (which is denoted with
		$\mathbf{x}\succeq \mathbf{y}$) if
		$$\sum_{i=1}^k x_{[i]} \le \sum_{i=1}^k y_{[i]} \text{ for all }k=1,...,p.$$
		\begin{enumerate}[(1).]
			\item If $\mathbf{x}_\uparrow \neq \mathbf{y}_\uparrow$, then $\mathbf{x}$
			is said to \emph{strictly weakly majorize} $\mathbf{y}$ (which is
			denoted with $\mathbf{x}\succ_w \mathbf{y}$).
			\item If $\mathbf{x}\succ_w \mathbf{y}$ and $\sum_{i=1}^p x_i = \sum_{i=1}^p y_i$, 
			$\mathbf{x}$
			is said to \emph{strictly majorize} $\mathbf{y}$ (which is
			denoted with $\mathbf{x}\succ \mathbf{y}$).
		\end{enumerate}
	\end{definition}

	Let $\mathbf{x}, \mathbf{y} \in N^p$.  We say that $\mathbf{x}$ covers $\mathbf{y}$ if $\mathbf{y} \prec \mathbf{x}$ 
	and there exists no vector $\mathbf{z}\in N^k$ such that $\mathbf{y} \prec \mathbf{z}\prec \mathbf{x}$. Namely, 
	there exists some $i<j$ such that $\mathbf{x}=\mathbf{y}+\mathbf{e_i}-\mathbf{e_j}$.

	\begin{lem}\cite{marshallInequalitiesTheoryMajorization2011}\label{lem:major}
		If $\alpha$ and $\beta$ are two nonincreasing integer sequences  in $N^{p}$ with $\alpha \prec \beta$,  then there exist $\mathbf{x}_1,\mathbf{x}_2,\dots, \mathbf{x}_l \in N^{k}$ such that $\alpha=\mathbf{x}_1, \beta=\mathbf{x}_l$ and  $\mathbf{x}_i$ is covered by $\mathbf{x}_{i+1}$ for $1\leq i<l$.
	\end{lem}

	Given positive integers numbers $m$ and $k$ with $m \geq k,  n = qk + r$ and $0 \leq  r < k$,   it is well known \cite{marshallInequalitiesTheoryMajorization2011}
	that the maximal and the minimal elements of the set
	\begin{equation*}S_m^k =\left\{ \mathbf{x\in }\text{ }%
	\mathbb{N}_{ +}
	^{k}:x_{1}\geq x_{2}\geq \ldots \geq x_{k}\geq 2, \sum_{i = 1}^{k}{x_i} = m\right\}
	\end{equation*}%
	with respect to the majorization order are respectively%
	\begin{equation*}
	\mathbf{x}^{\ast }\left(S_m^k \right) = (m - 2k)\mathbf{e}_{1} +2\mathbf{e} \text{ \ \ \ \ \
		\ \ \ \ and \ \ \ \ \ \ \ \ \ \ }\mathbf{x}_{\ast }\left(S_m^k \right)
	= q\mathbf{e}+ \mathbf{f},
	\end{equation*}
	where $\mathbf{e}$ be the all-ones vector of order $k$,  $\mathbf{e}_1=(1,0,\dots, 0)$ and $\mathbf{f}=(f_1,..., f_k)$ with $f_i=1$ for $1\leq i\leq r$ otherwise $f_i=0$.

	\section{The \texorpdfstring{$\alpha$}{α}-spectral radius of rose digraphs and tri-ring digraphs}

	For $m \geq k \geq 1$,  let $\mathbf{p}=(p_1, p_2, \ldots, p_k) \in S_m^{k}$  be a partition of $m$.
	We call the digraph $\bigodot\limits_{i=1}^kG_i(v_i)$ (depicted in Fig. \ref{Fig:Rose}) rose digraph  when  $G_i=\overrightarrow{C}_{p_i}$, which is denoted by $\Rose(\mathbf{p})$.
	When $k=2$, $\Rose(\mathbf{p})$ is $\infty$-digraph, written as $\infty(p_1, p_2)$.
	In \cite{MR2954483} and \cite{MR3777498}, the digraph $\Rose(\mathbf{p})$ is written as $\widetilde{\infty}(p_1-1, p_2-1, \ldots, p_k-1)$ and $\widetilde{\infty}_1(p_1-1, p_2-1, \ldots, p_k-1)$, respectively.

	\begin{figure}
		\begin{adjustbox}{center}
			\includegraphics[page=1,  trim=0 42 2 20,  clip, width=0.5\textwidth]{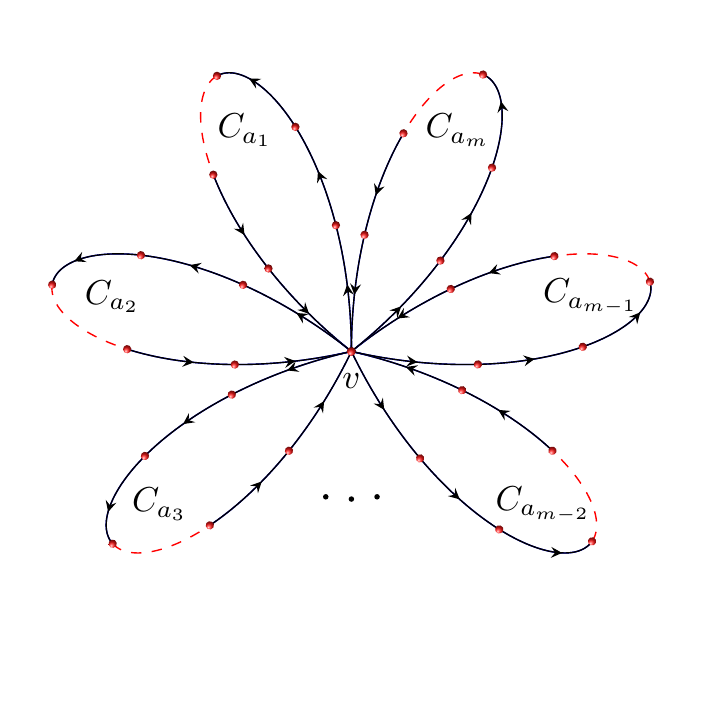}
		\end{adjustbox}
		\caption{\texorpdfstring{$\Rose(K)$}{Rose(K)} with \texorpdfstring{$K=(a_1, a_2, \cdots, a_m)$}{K=a1,a2,... }.} \label{Fig:Rose}
	\end{figure}

	Take $\gamma_1=\mathbf{x}^*(S_m^{k})$ and $\gamma_2=\mathbf{x}_*(S_m^{k})$
	to be the maximal and the minimal elements of a subset $S_m^{k} \subseteq
	\mathbb{R}^{n}$ with respect to the majorization order.  Denote $\mathcal{R}_{m}^k= \{\Rose(\mathbf{p}) \mid  \mathbf{p}
	\in S_m^{k}\}$.
	
	In \cite{MR2954483} and \cite{MR3777498}, Guo et al. and Li et al. showed that among  $\mathcal{R}_{m}^k$,  the digraph $\Rose(\mathbf{\gamma}_1)$ ($\Rose(\mathbf{\gamma}_2)$) is the unique digraph with  the maximal (minimal) adjacency spectral radius and signless Laplacian spectral radius, respectively.
	Next we determine the unique digraph with  the maximal (minimal) $\alpha$-spectral radius among  $\mathcal{R}_{m}^k$, respectively.
	
	\begin{lem}\label{lem3} Let $H=G_v(p, q)$ and $H'=G_v(p+1, q-1)$.
		If $p\ge q$,  then
		$$\rho_{\alpha}(H)< \rho_{\alpha}(H').$$
	\end{lem}
	
	\begin{proof}
		Let $d=\frac{1-\alpha}{x-\alpha}$. By Corollary \ref{cor:b1},  we have
		\begin{align*}
		\phi_{\alpha}(H)  & =(x-\alpha)^{p+q-2}\left(\phi_{\alpha}(G)-\psi_\alpha(G,v)\big(2\alpha+(1-\alpha)(d^{p-1}+d^{q-1})\big)\right), \\
		\phi_{\alpha}(H') & =(x-\alpha)^{p+q-2}\left( \phi_{\alpha}(G)-\psi_\alpha(G,v)\big(2\alpha+(1-\alpha)(d^{p}+d^{q-2})\big)\right).
		\end{align*}
		Since  $x>1$ and $p\ge q$, we have $0<d<1$ and
		\begin{align*}
		& \phi_{\alpha}(H')-\phi_{\alpha}(H)                                                   \\
		= & (x-\alpha)^{p+q-2}\psi_\alpha(G,v)(1-\alpha)(d^{p-1}+d^{q-1}-d^{p}-d^{q-2}) \\
		= & (x-\alpha)^{p+q-2}\psi_\alpha(G,v)(1-\alpha)(1-d)(d^{p-1}-d^{q-2 })<0.
		\end{align*}
		So,  $\rho_{\alpha}(H)< \rho_{\alpha}(H').$
		
	\end{proof}

	By Lemma \ref{lem3} and Lemma \ref{lem:major},  we have the following theorem.
	\begin{thm}\label{thm1}
		For $\mathbf{p},  \mathbf{q} \in S_m^{k}$,  if $\mathbf{p}  \preceq  \mathbf{q} $,
		then $\rho_{\alpha}(\Rose(\mathbf{p}))<\rho_{\alpha}(\Rose(\mathbf{q}))$.  Furthermore, $\Rose(\mathbf{\gamma}_1)$ and $\Rose(\mathbf{\gamma}_2)$  are the unique digraphs with the maximal and  minimal $\alpha$-spectral radius among $\mathcal{R}_{m}^k$, respectively.
	\end{thm}
	\begin{proof}
		By Lemma \ref{lem:major}, it is sufficient to show that $\rho_{\alpha}(\Rose(\mathbf{p}))<\rho_{\alpha}(\Rose(\mathbf{q}))$ when $\mathbf{q}$ covers  $\mathbf{p}$.
		Without loss of generality, let $\mathbf{q}=\mathbf{p}+\mathbf{e}_i-\mathbf{e}_j$ with $1\leq i<j\leq k$.
		Take $H=G_v(p_i,p_j)$  and $H'=G_v(q_i,q_j)=G_v(p_i+1,p_j-1)$, where $G=\Rose(\mathbf{\gamma})$, $\mathbf{\gamma}$ obtained from $\mathbf{p}$  by deleting the $i$-th and $j$-th entries of $\mathbf{p}$ and $v$ is the central vertex of $G$.
		By Lemma \ref{lem3}, we have $\rho_{\alpha}(H) < \rho_\alpha(H')$. Since $H'=\Rose(\mathbf{q})$ and $H=\Rose(\mathbf{p})$,
		then $\rho_{\alpha}(\Rose(\mathbf{p}))<\rho_{\alpha}(\Rose(\mathbf{q}))$ holds.
	\end{proof}

	Using the above theorem,  we immediately get the following result,  which is Theorem 3.2 in \cite{MR3872980}.
	\begin{cor}\label{thm4}
		$\infty(2, m-2)$  and $\infty(\lfloor{\frac{m}{2}}\rfloor,  \lceil{\frac{m}{2}}\rceil)$ are  the unique digraphs with the maximal  and  minimal $\alpha$-spectral radius, respectively, among   $\mathcal{R}_m^2$.
	\end{cor}
	
\begin{figure}
		\begin{adjustbox}{center}
			\includegraphics[page=3,  trim=0 0 0 5,  clip, width=0.8\textwidth]{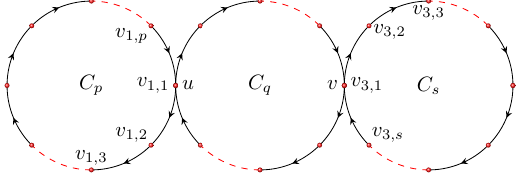}
		\end{adjustbox}
		\caption{\texorpdfstring{$\displaystyle\Inf_{u, v}(p, q, s)$}{Inf(p, q, s)}.}\label{fig:infty}
	\end{figure}

	Let $u, v$ be two distinct vertices of dicircle $\overrightarrow{C}_{q}$.
	Let $\displaystyle\Inf_{u, v}(p, q, s)$ be the digraph (as shown in Figure \ref{fig:infty})
	obtained from $\overrightarrow{C}_{q}$  by identifying vertex $u$ with a vertex of dicircle $\overrightarrow{C}_{p}$
	and identifying vertex $v$ with a vertex of dicircle $\overrightarrow{C}_{s}$ with $2 \leq p \leq s$.
	The digraph $\displaystyle\Inf_{u, v}(p, q, s)$ is called tri-ring digraph, which also is rewritten as $\Inf(p, q, s)$.
	Let $\displaystyle \INF(m)$  denote the set consisting of all tri-ring digraphs of size $m$.
	
	In \cite{MR3777498}, the digraph $\displaystyle\Inf_{u, v}(p, q, s)$ is denoted by $\widetilde{\infty}_2(p-1, q-2, s-1)$.
	In that paper,  Li et al. showed that the digraph $\Inf(2, 2, m-4)$ is the unique digraph with  the maximal
	signless Laplacian spectral radius among $\INF(m)$.
	
	In the following part of this section,  we will show that  $\Inf(2, 2, m-4)$ also is the digraph with the
	maximal $\alpha$-spectral radius among $\INF(m)$.

	From Lemma \ref{thmd1},  we can easily get the following lemma.
	\begin{lem}\label{lem:b6}Let $f_1, f_2, f_3$ be functions defined in Lemma \ref{thmd1}.
		If $\psi_\alpha(G,u)=\psi_\alpha(G,v)$, $\alpha f_3-f_1<0$ and $s\geq t\geq 2$,  we have $$\rho_{\alpha}(G_{u, v}(s, t))<\rho_{\alpha}(G_{u, v}(s+1, t-1)).$$
	\end{lem}
	
	\begin{proof}
		When $\psi_\alpha(G,u)=\psi_\alpha(G,v)$,  we have $f_1=f_2$ and
		\begin{align*}
		& \phi_{\alpha}(G_{u, v}(s+1, t-1))-\phi_{\alpha}(G_{u, v}(s, t))                                                 \\
		= & (\alpha f_3- f_1)\big( (x-\alpha)^{t -2}(1-\alpha)^{s+ 1}+(x -\alpha)^{s}(1-\alpha)^{t-1}                            \\
		- & (x-\alpha)^{t-1}(1-\alpha)^{s}-(x-\alpha)^{s-1}(1-\alpha)^{t}\big)                                                \\
		= & (\alpha f_3- f_1)(x-\alpha)^{s+t-1}(d^{t-1}-d^{s}) (1-d).
		\end{align*}			

		For all $x>1$, since $s \geq t \geq 2$ and $\alpha f_3-f_1<0$,  we have $1-d>0, d^{t-1}>d^{s}$ and 
		$$\phi_{\alpha}(G_{u, v}(s+1, t-1))-\phi_{\alpha}(G_{u, v}(s, t))<0,$$
		which implies $\rho_{\alpha}(G_{u, v}(s, t))<\rho_{\alpha}(G_{u, v}(s+1, t-1)).$
	\end{proof}

	By Lemma \ref{lem:b6},  we can easily get the following lemma.
	\begin{lem}\label{lem4}
		If $2\le p\le s$,  then $\rho_{\alpha}(\displaystyle\Inf_{u, v}(p, q, s))<\rho_{\alpha}(\displaystyle\Inf_{u, v}(p-1, q, s+1))$.
	\end{lem}
	\begin{proof} Take $G_{u, v}(s, t)$, $f_1,f_2,f_3$ be defined in Lemma \ref{thmd1}  with $s\ge p\ge 2$.
		When $G=\overrightarrow{C_q}$,  by direct calculation,  we have $\psi_\alpha(G,u)=\psi_\alpha(G,v)$ and
		\begin{align*}
		\alpha f_3-f_1 & =\alpha (x-\alpha)^{n-2}-(x-\alpha)^{n-1} \\
		& =-(x-\alpha)^{n -2}(x-2\alpha)<0.
		\end{align*}
		Then $\rho_{\alpha}(\displaystyle\Inf_{u, v}(p, q, s))<\rho_{\alpha}(\displaystyle\Inf_{u, v}(p-1, q, s+1))$ follows from Lemma \ref{lem:b6}.
	\end{proof}

	Let $C_p\cdot C_q\cdot G$ be the digraph obtained from digraphs $G,  \overrightarrow{C_p}$ and $\overrightarrow{C_q}$ by identifying a vertex of $\overrightarrow{C}_{p}$ with a vertex of $\overrightarrow{C}_{q}$ and identifying a vertex $u$ of $G$ with another vertex of $\overrightarrow{C}_{q}$.
	By Lemma \ref{lemd4},  we have all the digraphs $C_p\cdot C_q\cdot G$ have the same characteristic polynomials,  when identifying the vertex $u$ of digraphs $G$ and the distinct internal vertices in the same dicircle.

	\begin{thm}\label{thmd3}
		Let $H_1=C_p\cdot C_q\cdot G$ and $H_2=C_q\cdot C_p\cdot G$ be the digraphs defined above with $p> q\ge 2$. We have $$\rho_{\alpha}(H_1)>\rho_{\alpha}(H_2).$$
	\end{thm}
	\begin{proof}
		Without loss of generality,  let $w, w_1$ be two distinct vertices of dicircle $\overrightarrow{C}_{p}$ and $v, v_1$ be two distinct vertices of dicircle $\overrightarrow{C}_{q}$.  Take $H=\overrightarrow{C}_{p}w_1:v_1\overrightarrow{C}_{q}$.
		By Corollary \ref{cord2},  we have
		\begin{align*}
		\phi_{\alpha}(H_1)
		= & \psi_\alpha(H,v)(\phi_{\alpha}(G)-x\psi_\alpha(G,u))+\phi_{\alpha}(H)\psi_\alpha(G,u), \\
		\phi_{\alpha}(H_2)
		= & \psi_\alpha(H,w)(\phi_{\alpha}(G)-x\psi_\alpha(G,u))+\phi_{\alpha}(H)\psi_\alpha(G,u)
		\end{align*}
		and
		\begin{align*}
		\phi_\alpha(H)=& (x-\alpha )^{p+q-2}(x-2\alpha -(1-\alpha)(d^{p-1}+d^{q-1}))\\
		\psi_\alpha(H,v)
		=                          & 2(x-\alpha)^{p+q-2}-(x-\alpha)^{q-2}(1-\alpha)^{p}-x(x-\alpha)^{p+q-3}, \\
		\psi_\alpha(H,w)= & 2(x-\alpha)^{p+q-2}-(x-\alpha)^{p-2}(1-\alpha)^{q}-x(x-\alpha)^{p+q-3}.
		\end{align*}
		Then we have
		\begin{align*}
		& \phi_{\alpha}(H_2,x)-\phi_{\alpha}(H_1,x)                                                                     \\
		= & (\psi_\alpha(H,w)-\psi_\alpha(H,v))(\phi_{\alpha}(G)-x\psi_\alpha(G,u))\:          \\
		= & ((x-\alpha)^{q-2}(1-\alpha)^{p}-(x-\alpha)^{p-2}(1-\alpha)^{q})(\phi_{\alpha}(G)-x\psi_\alpha(G,u))\\
		= & (x-\alpha)^{p+q-2}(d^{p}-d^{q})(\phi_{\alpha}(G)-x\psi_\alpha(G,u)).
		\end{align*}
		By Lemma \ref{lemz2}, we have
		$\phi_{\alpha}(G)-x\psi_\alpha(G,u)<0$ for all $x\ge \rho_{\alpha}(G)$.  
		For $x>1$, $d^p-d^q<0$ follows from $p>q$. So $\phi_{\alpha}(H_2,x)-\phi_{\alpha}(H_1,x)>0$ holds when $x\ge \rho_{\alpha}(G)$. From (ii) of Lemma \ref{lemc2}, $\min(\rho_\alpha(H_1),\rho_\alpha(H_2))>\rho_\alpha(G)$. Then $\rho_{\alpha}(H_1)>\rho_{\alpha}(H_2)$ follows.
	\end{proof}

	By Theorem \ref{thmd3} and  Lemma \ref{lem4},  we have the following theorem.
	
	\begin{thm}\label{thmd4}
		$\Inf(2, 2, m-4)$ is the unique digraphs with  the maximal  $\alpha$-spectral radius among  $\displaystyle \INF(m)$ for $m \geq 6$.
	\end{thm}
	\begin{proof}
		Suppose that $G=\Inf(p,q,s)$ ($2 \leq p \leq s$) is the digraph with  maximal $\alpha$-spectral radius among  $\displaystyle \INF(m)$. By  Lemma \ref{lem4},  we have $p=2$.  From Theorem \ref{thmd3}, we have $p \geq q$. 
		Sine $p \geq 2$, we have that  $G=\Inf(2, 2, m-4)$ is the unique digraph with  maximal $\alpha$-spectral radius among  $\displaystyle \INF(m)$. 
	\end{proof}
	
	Using the above theorem,  when $\alpha=\frac{1}{2}$,  we can immediately solve the Problem $2.8$ about  maximal spectral radius in \cite{MR3777498}.
	The digraph $\Inf(2, 2, m-4)$ is the unique digraph with  the maximal
	signless Laplacian spectral radius among $\displaystyle \INF(m)$.
	
	\begin{remark}
		The minimal spectral radius part of Problem $2.8$ in  \cite{MR3777498} is very hard to solve. By numerical computation, one can see that  it is difficult to determine the digraph with minimal signless Laplacian spectral radius among $\displaystyle \INF(m)$.  It can be seen from Table 1 that the pattern of digraphs with the minimal signless Laplacian spectral radius is related to size $m$. Furthermore,  for fixed size $m$, the pattern of digraphs with the minimal $\alpha$-spectral radius is related to $\alpha$.
	\end{remark}
	In the following table, we list  the digraphs with first four minimal $\alpha$-spectral radius  with given size and $\alpha$. 
	
	\begin{table}[!ht]\label{Tab1}
		\begin{adjustbox}{width=\textwidth,center}
			\tabcolsep=2pt 
			\extrarowsep=1mm
			\begin{tabu} to 1.1\textwidth {|X[c,m] ||
					X[1.1,c,m] |X[0.8,c,m] ||
					X[1.1,c,m] |X[0.8,c,m] ||
					X[1.1,c,m] |X[0.8,c,m] ||
					X[1.1,c,m] |X[0.8,c,m] |}
				\tabucline -
				$\alpha$ | $m$ &		
				$G_1$	&	$\rho_\alpha(G_1)$ 	&
				$G_2$	&	$\rho_\alpha(G_2)$ 	&
				$G_3$	&	$\rho_\alpha(G_3)$ 	&
				$G_4$	&	$\rho_\alpha(G_4)$    \\\tabucline -
				0.5 | 15 &		
				$\Inf$(5,5,5) 	&	1.312 	&
				$\Inf$(6,6,3) 	&	1.316 	&
				$\Inf$(6,5,4) 	&	1.319 	&
				$\Inf$(7,5,3) 	&	1.341    \\\tabucline -
				0.5 | 24 &	
				$\Inf$(9,9,6) 	&	1.208 	&
				$\Inf$(8,8,8) 	&	1.213 	&	
				$\Inf$(9,8,7) 	&	1.215 	&
				$\Inf$(10,9,5) 	&	1.217    \\\tabucline -
				0.2 | 15 &	
				$\Inf$(5,5,5) 	&	1.233   &
				$\Inf$(6,5,4) 	&	1.2381 	&
				$\Inf$(6,6,3) 	&	1.2383 	&
				$\Inf$(7,5,3) 	&	1.255   \\\tabucline -
				0.8 | 15 &		
				$\Inf$(6,6,3) 	&	1.619 	&
				$\Inf$(5,5,5) 	&	1.624 	&
				$\Inf$(6,5,4) 	&	1.625   &	
				$\Inf$(7,5,3) 	&	1.63    \\\tabucline -
			\end{tabu}
		\end{adjustbox}
		\caption{Digraphs with small $\alpha$-spectral radius in $\INF(m)$ for some $\alpha$ and $m$.}
	\end{table}

	\section{The \texorpdfstring{$\alpha$}{α}-spectral radius of digraphs in $\widetilde{\boldsymbol\Theta}_k(m)$}
	
	For integer $t\geq 1$ and $m \geq 2t-1$, take
	\begin{equation*}\widetilde{S}_m^t =\left\{ \mathbf{x}\in 
	\mathbb{N}_{ +}
	^{k}:x_{1}\geq x_{2}\geq \ldots \geq x_{t}\geq 1, \sum_{i = 1}^{t}{x_i} = m \text{ and } x_{t-1}\geq 2\right\}
	\end{equation*}
	and $\displaystyle \widehat{S}^{t}_{m}=\bigcup_{i=2t-1}^{m}\widetilde{S}^t_i$.
	Let $\mathbf{x}^{\ast }\left(\widetilde{S}_m^t \right)=\mathbf{\zeta}^m_t=(m-2t+3, \overbrace{2,\ldots,2}^{t-2},1)$ and $\mathbf{x}_{\ast }\left(\widetilde{S}_m^t \right)=(l_1,l_2,\dots,l_t)$ be the maximal and the  minimal elements of $\widetilde{S}_m^t$ with respect to the majorization order, respectively. Clearly, $l_1-l_t \leq 1$.
	
	Denote $\mathbf{\zeta}_t=(\overbrace{2,\ldots,2}^{t-1},1)$. 
	It is not difficult to check that $\mathbf{\zeta}^m_t$ and $\mathbf{\zeta}_t$ be the maximal and the  minimal elements of $\displaystyle \widehat{S}^{t}_{m}$ with respect to the weakly majorization order, respectively.

	For $m \geq 4, k\geq 3$, take $$Q(m,k)=\{(\mathbf{K}_1,\mathbf{K}_2)| \mathbf{K}_1\in \widetilde{S}^s_{m_1}, \mathbf{K}_2\in \widetilde{S}^t_{m_2},s+t=k, m_1+m_2=m, s \geq t\}.$$ 
	
	For $\mathbf{K}_1=(k_1,k_2,\dots,k_s)\in \widetilde{S}^s_{m_1}$ and $\mathbf{K}_2=(l_1,l_2,\dots,l_t)\in \widetilde{S}^t_{m_2}$, denote by $\widetilde{\Theta}(\mathbf{K}_1,\mathbf{K}_2)$ the simple digraph with $m_1+m_2$ edges obtained from two vertices $u$, $v$ by adding $s$ internally-disjoint directed paths with lengths $k_1,k_2,\ldots, k_s$ from $u$ to $v$ and $t$  internally-disjoint directed paths with lengths $l_1,l_2,\ldots, l_t$ from $v$ to $u  $  (as shown in Figure~\ref{Figtheta}). 
	The simple digraph $\widetilde{\Theta}(\mathbf{K}_1,\mathbf{K}_2)$ is called generalized theta digraph, which also is denoted by $\widetilde{\Theta}(k_1, k_2, \ldots, k_s;l_1, l_2, \ldots, l_t)$. 
	For fixed positive integers $s,t,m_1,m_2$, 
	let $\prescript{s}{t}{\widetilde{\boldsymbol\Theta}_{m_2}^{m_1}}$  denote the set consisting of all digraphs $\widetilde{\Theta}(\mathbf{K}_1,\mathbf{K}_2)$  for $\mathbf{K}_1 \in \widetilde{S}_{m_1}^{s}$ and $\mathbf{K}_2 \in \widetilde{S}_{m_2}^{t}$.
	
	For the sake of brevity, the digraph $\widetilde{\Theta}(\mathbf{K}_1,\mathbf{K}_2)$ is rewritten as  $\widetilde{\Theta}_m(\mathbf{K}_1)$ when the length of $K_2$ equals to one.
	In the sequel, we denote  
	$\prescript{s}{t}{\widetilde{\boldsymbol\Theta}(m)}=\displaystyle\bigcup_{\mathclap{\substack{m_1+m_2=m}}}\prescript{s}{t}{\widetilde{\boldsymbol\Theta}_{m_2}^{m_1}}$, $\widetilde{\boldsymbol\Theta}_k^{1}(m)=\prescript{k-1}{1}{\widetilde{\boldsymbol\Theta}(m)}$  and $\widetilde{\boldsymbol\Theta}_k(m)=\displaystyle\bigcup_{\mathclap{\substack{s+t=k\\m_1+m_2=m}}}\prescript{s}{t}{\widetilde{\boldsymbol\Theta}_{m_2}^{m_1}}$. It is obvious that $\widetilde{\boldsymbol\Theta}_k(m)$ is the set of all generalized theta digraphs of size $m$.

	\begin{figure}
		\begin{adjustbox}{
				height=2.5cm,
				center}
			\includegraphics[page=2,  trim=0 0 0 5,  clip, width=0.8\textwidth]{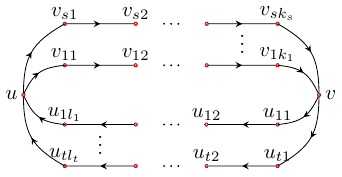}
		\end{adjustbox}
		\caption{ $\widetilde{\Theta}(\mathbf{K}_1, \mathbf{K}_2)$}\label{Figtheta}
	\end{figure}
	
	Let $U,V$ be two vertex sets of digraph $G$.  Write $\Lambda(G,\alpha)=x I_n -A_{\alpha}(G)$ and denote $\Lambda(G)(U|V)$ be the submatrix of $\Lambda(G,\alpha)$ obtained by deleting the rows corresponding to vertices of $U$ and the columns corresponding to vertices of $V$.
	When $U=\{u\}, V=\{v\}$, we write $\Lambda(G)(U|V)$ as $\Lambda(G)(u|v)$, respectively.
	When $U=V$, we write $\Lambda(G)(U|V)$ as $\Lambda(G)(U)$.
	\begin{lem}\label{lemd5}
		Let $G_1,G_2$ be digraphs with order $n_1$ and $n_2$. Suppose that $u_1,v_1$ are vertices of $G_1$ with $d^{-}_{u_1}(G_1)=d^{+}_{v_1}(G_1)=0$ and $u_2,v_2$ are vertices of $G_2$ with $d^{+}_{u_2}(G_2)=d^{-}_{v_2}(G_2)=0$.  Let $H$ be the digraph obtained from $G_1,G_2$ by identifying/merging  vertices $u_1,u_2$ and  $v_1,v_2$ as new vertices $u$ and $v$, respectively.
		Then
		\begin{align*}
		\phi_\alpha(H,x)= & (x-\alpha d_1) (x-\alpha d_2)\det(A_1)  \det(A_2)         \\
		-                 & \det(\Lambda(G_1)(v_1|u_1))  \det(\Lambda(G_2)(u_2|v_2)),
		\end{align*}
		where $A_1=\Lambda(G_1)(\{u_1,v_1\})$, $A_2=\Lambda(G_2)(\{u_2,v_2\})$ and  $d_1=d^{ +}_{u_1}(G_1)$, $d_2=d^{ +}_{v_2}(G_2)$.
	\end{lem}
	
	\begin{proof} 
		Giving a suitable ordering for the vertices of $G_1,G_2$ and $H$, we can assume that $\Lambda(G_1), \Lambda(G_2),\Lambda(H)$ has the following form:
		$$\Lambda(G_1)=\begin{pmatrix} x-\alpha d_{1} & 0 & \alpha_{1} \\ 0 & x & 0 \\ 0 & \beta_{1} & A_1\end{pmatrix}, \qquad
		\Lambda(G_2)=\begin{pmatrix} x & 0 & 0 \\ 0 & x-\alpha d_{2} & \alpha_{2} \\ \beta_{2} & 0 & A_2 \end{pmatrix},$$
		$$\Lambda(H)=\begin{pmatrix}x-\alpha d_{1} & 0 & \alpha_{1} & 0 \\ 0 & x-\alpha d_{2} & 0 & \alpha_{2} \\ 0 & \beta_{1} & A_1 & 0 \\ \beta_{2} & 0 & 0 & A_2\end{pmatrix}.$$
		
		By Laplace expansion with respect to the first two rows of $\Lambda(H)$,  we have
		
		\[
		\begin{aligned}
		\phi_\alpha(H) = (x-\alpha d_1) (x-\alpha d_2)\det(A_1)  \det(A_2)
		+\begin{vmatrix}0 & 0 & \alpha_{1} & 0 \\ 0 & 0 & 0 & \alpha_{2} \\ 0 & \beta_{1} & A_1 & 0 \\ \beta_{2} & 0 & 0 & A_2\end{vmatrix}.
		\end{aligned}
		\]
		By
		$$\begin{vmatrix}
		0         & 0         & \alpha_{1} & 0          \\
		0         & 0         & 0          & \alpha_{2} \\
		0         & \beta_{1} & A_1        & 0          \\
		\beta_{2} & 0         & 0          & A_2
		\end{vmatrix}
		=-\begin{vmatrix}
		0       & \alpha_{1} & 0         & 0          \\
		\beta_1 & A_1        & 0         & 0          \\
		0       & 0          & 0         & \alpha_{2} \\
		0       & 0          & \beta_{2} & A_2
		\end{vmatrix}$$
		and
		$$\Lambda(G_1)(v_1|u_1)
		=\begin{pmatrix}
		0       & \alpha_{1} \\
		\beta_1 & A_1
		\end{pmatrix}, \qquad
		\Lambda(G_2)(u_2|v_2)
		=\begin{pmatrix}
		0       & \alpha_{2} \\
		\beta_2 & A_2
		\end{pmatrix},$$
		we have
		\begin{align*}
		\phi_\alpha(H,x)= & (x-\alpha d_1) (x-\alpha d_2)\det(A_1)  \det(A_2)         \\
		-                 & \det(\Lambda(G_1)(v_1|u_1))  \det(\Lambda(G_2)(u_2|v_2)).
		\end{align*} \qedhere
	\end{proof}
	
	Let $A$ be an $n \times n$ matrix  and  $D(A)$ be the Coates digraph (See \cite{belardoCombinatorialApproachComputing2010a}) associated with $A$.
	\begin{lem}\cite{maybeeMatricesDigraphsDeterminants1989}\label{lem:cofactor}
		For any $a_{i j}$ off-diagonal  element of $A$,  the cofactor of $a_{i j}$ is given by
		$$
		A_{i j}=\sum_{k}(-1)^{l_{k}} A\left[p_{k}(j \rightarrow i)\right] \operatorname{det} A\left[V\left(p_{k}\right)\right],
		$$
		where the sum is taken over all paths in $D(A)$ from $j$ to $i,$ and $l_{k}$ is the length of path $p_{k}$.
		
	\end{lem}
	
	Let $\mathbf{K}=(k_1,k_2,\ldots, k_s) \in \widetilde{S}_{m}^s$ and  $P_{u\to v}(\mathbf{K})$ be the digraph with $n$ vertices and $m(m=n+s-2)$ edges obtained from two vertices $u$, $v$ by adding $s$ internally-disjoint directed paths with lengths $k_1,k_2,\ldots, k_s$ from $u$ to $v$.  Let $d=\frac{1-\alpha}{x-\alpha}$.
	Take matrix $A$ in Lemma \ref{lem:cofactor} to be $\Lambda(P_{u\to v}(\mathbf{K}))$, we have
	\begin{equation}\label{det:cofactor}
	\det(\Lambda(P_{u\to v}(\mathbf{K}))(v|u))= (x-\alpha)^{n-1} \sum_{i=1}^{s} d^{k_i}.
	\end{equation}
	
	By Lemma \ref{lemd5} and Formula \eqref{det:cofactor}, we have the following corollary.
	
	\begin{cor}\label{cors1} 
		Let $\mathbf{K}_1=(k_1,k_2,\dots,k_s)$, $\mathbf{K}_2=(l_1,l_2,\dots,l_t)$ and the digraph $\widetilde{\Theta}
		(\mathbf{K}_1,\mathbf{K}_2)$ with $n$ vertices, then we have
		\begin{equation}\label{theta:chp}
		\begin{aligned}
		\phi_{\alpha}(\widetilde{\Theta}
		(\mathbf{K}_1,\mathbf{K}_2))
		=(x-\alpha)^{n-2}(x-\alpha s)(x-\alpha t)-
		(x-\alpha)^n(\sum_{i=1}^sd^{k_i}\sum_{j=1}^t
		d^{l_j}).
		\end{aligned}
		\end{equation}
	\end{cor}
	\begin{proof}
		Take $H=\widetilde{\Theta}(\mathbf{K}_1, \mathbf{K}_2)$, $G_1=P_{u_1\to v_1}(\mathbf{K}_1)$, $G_2=P_{v_2\to u_2}(\mathbf{K}_2)$ and $|V(G_i)|=n_i(i=1,2)$ in Lemma \ref{lemd5},
		then we have
		\begin{align*}
		\phi_\alpha(H)= & (x-\alpha s) (x-\alpha t)\det(A_1)  \det(A_2)             \\
		-               & \det(\Lambda(G_1)(v_1|u_1))  \det(\Lambda(G_2)(u_2|v_2)),
		\end{align*}
		where $\det(A_1)=(x-\alpha)^{n_1-2}$ and $\det(A_2)=(x-\alpha)^{n_2-2}$.
		
		Take $\mathbf{K}=\mathbf{K}_1$ in Formula \eqref{det:cofactor}, we have
		\begin{align*}
		\det(\Lambda(G_1)(v_1|u_1)) =(x-\alpha)^{n_1-1} \sum_{i=1}^{s} d^{k_i}.
		\end{align*}
		Similarly, we have
		\begin{align*}
		\det(\Lambda(G_2)(u_2|v_2)) =(x-\alpha)^{n_2-1} \sum_{i=1}^{t} d^{l_i}.
		\end{align*}
		Then
		\begin{align*}
		\phi_\alpha(H)= & (x-\alpha s) (x-\alpha t)(x-\alpha)^{n_1-2}(x-\alpha)^{n_2-2}                                             \\
		-               & (x-\alpha)^{n_1-1} \sum_{i=1}^{s} d^{k_i}(x-\alpha)^{n_2-1} \sum_{i=1}^{t} d^{l_i}                        \\
		=               & (x-\alpha s) (x-\alpha t)(x-\alpha)^{n-2}-(x-\alpha)^{n} (\sum_{i=1}^{s} d^{k_i} \sum_{i=1}^{t} d^{l_i}),
		\end{align*}
		where $n=n_1+n_2-2$.
		
	\end{proof}

	In the sequel, denote $\displaystyle F(\mathbf{K}_1,\mathbf{K}_2)=\sum_{i \in \mathbf{K}_1}d^i\sum_{j \in \mathbf{K}_2}d^j$ for $(K_1,K_2) \in Q(m,k)$ and $d=\frac{1-\alpha}{x-\alpha}$.
	\begin{lem}\label{lems111}
		For $\mathbf{K}_1\in \widetilde{S}_{m_1}^{s}, \mathbf{K}'_1 \in \widetilde{S}_{m'_1}^{s},\ \mathbf{K}_2\in \widetilde{S}_{m_2}^{t},  \mathbf{K}'_2 \in \widetilde{S}_{m'_2}^{t}$ and $m_1+m_2=m'_1+m'_2=m$.  If $F(\mathbf{K}_1,\mathbf{K}_2)\le F(\mathbf{K}'_1,\mathbf{K}'_2) $, then we have
		$$\rho_{\alpha}(\widetilde{\Theta}(\mathbf{K}_1,\mathbf{K}_2))\le \rho_{\alpha}(\widetilde{\Theta}(\mathbf{K}'_1,\mathbf{K}'_2)),$$
		and the equality holds if and only if $\mathbf{K}_1 =\mathbf{K}'_1$ and $\mathbf{K}_2 =\mathbf{K}'_2 $.
	\end{lem}
	\begin{proof}
		By Corollary \ref{cors1}, we have
		\begin{align*}
		\phi_{\alpha}(\widetilde{\Theta}
		(\mathbf{K}_1,\mathbf{K}_2))
		=(x-\alpha)^{n-2}(x-\alpha s)(x-\alpha t)-
		(x-\alpha)^nF(\mathbf{K}_1,\mathbf{K}_2).
		\end{align*}
		Hence,
		\begin{align*}
		& \phi_{\alpha}(\widetilde{\Theta}(\mathbf{K}'_1,\mathbf{K}'_2))-\phi_{\alpha}(\widetilde{\Theta}(\mathbf{K}_1,\mathbf{K}_2)) \\
		& =(x-\alpha)^n(F(\mathbf{K}_1,\mathbf{K}_2)-F(\mathbf{K}'_1,\mathbf{K}'_2)).
		\end{align*}
		If $F(\mathbf{K}_1,\mathbf{K}_2)\le F(\mathbf{K}'_1,\mathbf{K}'_2) $, then  $\phi_{\alpha}(\widetilde{\Theta}(\mathbf{K}'_1,\mathbf{K}'_2))\le \phi_{\alpha}(\widetilde{\Theta}(\mathbf{K}_1,\mathbf{K}_2))$, which implies $\rho_{\alpha}(\widetilde{\Theta}(\mathbf{K}_1,\mathbf{K}_2))\le \rho_{\alpha}(\widetilde{\Theta}(\mathbf{K}'_1,\mathbf{K}'_2)),$ and the equality holds if and only if $\mathbf{K}_1 =\mathbf{K}'_1$ and $\mathbf{K}_2 =\mathbf{K}'_2 $.
	\end{proof}

	\begin{lem}\label{lems1}
		For $\mathbf{K}_1, \mathbf{K}'_1 \in \widetilde{S}_{m_1}^{s},\ \mathbf{K}_2,  \mathbf{K}'_2 \in \widetilde{S}_{m_2}^{t}$.  If $\mathbf{K}_1 \preceq \mathbf{K}'_1$ and $\mathbf{K}_2 \preceq \mathbf{K}'_2 $, then we have
		$$\rho_{\alpha}(\widetilde{\Theta}(\mathbf{K}_1,\mathbf{K}_2))\le \rho_{\alpha}(\widetilde{\Theta}(\mathbf{K}'_1,\mathbf{K}'_2)).$$
		Moreover, the equality holds if and only if $\mathbf{K}_1 =\mathbf{K}'_1$ and $\mathbf{K}_2 =\mathbf{K}'_2 $.
	\end{lem}
	\begin{proof}
		
		If $\mathbf{K}_1 =\mathbf{K}'_1$ and $\mathbf{K}_2 =\mathbf{K}'_2 $, then $\widetilde{\Theta}(\mathbf{K}_1,\mathbf{K}_2)=\widetilde{\Theta}(\mathbf{K}'_1,\mathbf{K}'_2)$ and $\rho_{\alpha}(\widetilde{\Theta}(\mathbf{K}_1,\mathbf{K}_2))= \rho_{\alpha}(\widetilde{\Theta}(\mathbf{K}'_1,\mathbf{K}'_2)).$
		
		If one of  $\mathbf{K}_1 \preceq \mathbf{K}'_1$ and $\mathbf{K}_2 \preceq \mathbf{K}'_2 $ is strict, 
		from $F(\mathbf{K}_1,\mathbf{K}_2)=F(\mathbf{K}_2,\mathbf{K}_1)$, 
		without loss of generality, we assume that $\mathbf{K}_1 \prec \mathbf{K}'_1$ and $\mathbf{K}_2 = \mathbf{K}'_2$. By Lemmas \ref{lem:major} and \ref{lems111}, to prove $\rho_{\alpha}(\widetilde{\Theta}(\mathbf{K}_1,\mathbf{K}_2))< \rho_{\alpha}(\widetilde{\Theta}(\mathbf{K}'_1,\mathbf{K}'_2))$, it is sufficient to show 	\begin{center}
			$F(\mathbf{K}_1,\mathbf{K}_2)< F(\mathbf{K}'_1,\mathbf{K}_2)$, when $\mathbf{K}'_1$ covers  $\mathbf{K}_1$.
		\end{center}
		Without loss of generality, suppose that $\mathbf{K}'_1=\mathbf{K}_1+\mathbf{e}_i-\mathbf{e}_j$ with $1\leq i<j\leq s$.
		
		Take $p=\sum\limits_{r=1}^s d^{k_r}-d^{k_i}-d^{k_j}$ and $q=\sum\limits_{j=1}^t d^{l_j}$. We have
		\begin{align*}
		& F(\mathbf{K}'_1,\mathbf{K}_2)-F(\mathbf{K}_1,\mathbf{K}_2) \\
		& =(p+d^{k_i+1}+d^{k_j-1})q-(p+d^{k_i}+d^{k_j})q             \\
		& =(d^{k_i}-d^{k_j-1})(d-1)q.
		\end{align*}
		Since $d=\frac{1-\alpha}{x-\alpha}$, $x>1$ and  $k_i\ge k_j$,  $F(\mathbf{K}'_1,\mathbf{K}_2)> F(\mathbf{K}_1,\mathbf{K}_2)$ holds.
		The conclusion of this lemma follows.
	\end{proof}

	By Lemma \ref{lems1}, we have the following result.
	\begin{thm}\label{thmd8}
		Let $\mathbf{\xi}_1=\mathbf{x}^{\ast }(\widetilde{S}_{m_1}^s),\mathbf{\xi}_2=\mathbf{x}^{\ast }(\widetilde{S}_{m_2}^t)$ and $\mathbf{\eta}_1=\mathbf{x}_{\ast }(\widetilde{S}_{m_1}^s), \mathbf{\eta}_2=\mathbf{x}_{\ast }(\widetilde{S}_{m_2}^t)$.
		
		Then
		$\widetilde{\Theta}(\mathbf{\xi}_1,\mathbf{\xi}_2)$ and $\widetilde{\Theta}(\mathbf{\eta}_1,\mathbf{\eta}_2)$  are  the unique 	digraphs with the maximal and minimal $\alpha$-spectral radius among $\prescript{s}{t}{\widetilde{\boldsymbol\Theta}_{m_2}^{m_1}}$, respectively.
	\end{thm}

	In \cite{MR2954483} and \cite{MR3777498},  Guo et al. and Li et al. showed that the digraph
	$\widetilde{\Theta}_m(\mathbf{\zeta}^m_t)$
	($\widetilde{\Theta}_m(\mathbf{\zeta}_t)$)
	achieves the maximal (minimal) adjacency spectral radius and signless Laplacian spectral radius among $\widetilde{\boldsymbol{\Theta}}_t^1(m)$ and when $s\ge t\ge 2$ and $m_1\ge m_2$, the digraph $\widetilde{\Theta}(\mathbf{\zeta}^{m-2t+1}_s,\mathbf{\zeta}_t)$ is the unique digraph with  the maximal adjacency spectral radius and signless Laplacian spectral radius among $\prescript{s}{t}{\widetilde{\boldsymbol\Theta}(m)}$,  respectively.
	In the following,  we will show that the above conclusions for $\alpha$-spectral radius also hold.

	\begin{thm}\label{thm2}Take $s\ge t\ge 1$, $s+t=k$, $m \geq 2k-2$ and $m'=m-2t+1$. Then the following hold.
		\begin{enumerate}[(1).]
			\item	$\widetilde{\Theta}(\mathbf{\zeta}^{m'}_s,\mathbf{\zeta}_t)$ is the unique digraph with the maximal $\alpha$-spectral radius among $\prescript{s}{t}{\widetilde{\boldsymbol\Theta}(m)}$.
			\item When $t=1$, $\widetilde{\Theta}(\mathbf{\zeta}_t)$  is the unique digraph with the minimal $\alpha$-spectral radius among $\prescript{s}{t}{\widetilde{\boldsymbol\Theta}(m)}$.
		\end{enumerate}
	\end{thm}
	\begin{proof}       
		\noindent 
		Suppose $\widetilde{\Theta}(\mathbf{K}_1,\mathbf{K}_2)$ is the digraph with the maximal $\alpha$-spectral radius among $\prescript{s}{t}{\widetilde{\boldsymbol\Theta}(m)}$. By Theorem \ref{thmd8}, $\mathbf{K}_1=\mathbf{\zeta}^{m_1}_{s},\mathbf{K}_2=\mathbf{\zeta}^{m_2}_{t}$ for some $m_1,m_2$ with $m_1+m_2=m$. Assume that $m_2-(2t-1)=l_1>0$. Set $p=(s-2)d^2+d$, $q=(t-2)d^2+d$ and $k=s+t$, we have 
		\begin{align*}
		&F(\mathbf{\zeta}^{m-2t+1}_s,\mathbf{\zeta}_t)-F(\mathbf{K}_1,\mathbf{K}_2)\\
		=&
		(d^{m-2(k-2)}+p)(d^{2}+q)-(d^{m-2(k-2)-l_1}+p)(d^{2+l_1}+q)\\
		=&(pd^2-qd^{m-2(k-2)-l_1})(1-d^{l_1}).
		\end{align*}
		Since $s \geq t$,  $p \geq q$.  From $m-2(k-2)-l_1\geq 2$, we have $$F(\mathbf{\zeta}^{m-2t+1}_s,\mathbf{\zeta}_t)>F(\mathbf{K}_1,\mathbf{K}_2).$$
		From Lemma \ref{lems111}, we have $$\rho_{\alpha}(\widetilde{\Theta}(\mathbf{\zeta}^{m-2t+1}_s,\mathbf{\zeta}_t))>\rho_{\alpha}(\widetilde{\Theta}(\mathbf{K}_1,\mathbf{K}_2)),$$ 
		a contradiction with assumption. Hence, $m_2=2t-1$ and the conclusion of (1) follows.

		(2) Suppose $\widetilde{\Theta}(\mathbf{K}_1,\mathbf{K}_2)$ is the digraph with the minimal $\alpha$-spectral radius among $\widetilde{\boldsymbol{\Theta}}_k^1(m)$. Assuming $\mathbf{K}_1 \neq \mathbf{\zeta}_{k-1}$, then we have $a_1\geq 3$. Take $\mathbf{K}_1^{''}=(a_1-1,a_2,\dots,a_{k-1}) \in \widetilde{S}_{m_1-1}^{k-1}$,
		$\mathbf{K}_2^{''}=(m_2+1)$.  We have
		\begin{align*}
		F(\mathbf{K}_1,\mathbf{K}_2)-F(\mathbf{K}_1^{''},\mathbf{K}_2^{''})=&(p+d^{a_1})d^{m_2}-(p+d^{a_1-1})d^{m_2+1}\\
		=&p(d^{m_2}-d^{m_2+1})>0.
		\end{align*}
		From Lemma \ref{lems111}, we have  $\rho_{\alpha}(\widetilde{\Theta}_m(\mathbf{K}_1))>\rho_{\alpha}(\widetilde{\Theta}_m(\mathbf{\zeta}_{k-1}))$, a contradiction with assumption.
		The assertion of (2) holds.
	\end{proof}
	
	\begin{remark}
		For $t \geq 2$,  by numerical computation (see Table 2, one can see that  the digraph with minimal $\alpha$-spectral radius among $\prescript{s}{t}{\widetilde{\boldsymbol\Theta}(m)}$ is related to the value of $\alpha$.
	\end{remark}
	\begin{table}\label{table2}
		\tabulinestyle{on 2pt gray}
		\begin{tabu} to 0.5\textwidth {|X[0.5,c,m]|X[c,m]|X[4,c,m]|}
			\tabucline - &	\multicolumn{2}{|c|}{$\alpha=0.6$}\\\tabucline -
			i & $\rho_\alpha(G)$ & ($K_1,K_2$)\\\tabucline -	
			1&  2.424 & ((2, 2, 2, 2), (4, 3, 3))\\\tabucline - 
			2 & 2.426 & ((3, 2, 2, 2), (3, 3, 3))\\\tabucline - 
			3 & 2.429 & ((2, 2, 2, 1), (4, 4, 3))\\\tabucline - 
			4 & 2.4317 & ((3, 3, 3, 3), (2, 2, 2))\\\tabucline - 
		\end{tabu}
		\begin{tabu} to 0.48\textwidth {|X[c,m]|X[4,c,m]|}
			\tabucline -	
			\multicolumn{2}{|c|}{$\alpha=0.2$} \\\tabucline -
			$\rho_\alpha(G)$ & ($K_1,K_2$)\\\tabucline -			
			1.6971 & ((2, 2, 2, 1), (4, 4, 3))\\\tabucline - 
			1.6985 & ((2, 2, 2, 2), (4, 3, 3))\\\tabucline - 
			1.7095 & ((3, 2, 2, 2), (3, 3, 3))\\\tabucline - 
			1.7227 & ((2, 2, 2, 1), (5, 3, 3))\\\tabucline - 
		\end{tabu}
		\caption{Digraphs with the first four smallest $\alpha$-spectral radius in $\prescript{4}{3}{\widetilde{\boldsymbol\Theta}(18)}$ for  $\alpha=0.6, 0.2$.}
	\end{table}

	From $(1)$ of Theorem \ref{thm2},  we can immediately get the following corollary,  which is Theorem $3.1$ in \cite{MR3872980}.
	\begin{cor}
		$\widetilde{\Theta}(m-2, 1;1)$ and $\widetilde{\Theta}(2, 1;m-3)$ are the unique digraphs with the maximal and minimal $\alpha$-spectral radius 
		among all $\widetilde{\Theta}(k_1, k_2;l_1)$-digraphs of size $m$, respectively.
	\end{cor}

	In the following part of this section, we will show that the patterns of digraphs with the minimal and maximal $\alpha$-spectral radius among $\widetilde{\boldsymbol{\Theta}}_k(m)$ are related with the value of $\alpha$.

	\begin{thm}Let $k \geq 3$ and $m\geq 2k-2$. $s'=\lceil \frac{k}{2} \rceil, t'=\lfloor \frac{k}{2} \rfloor$. Then there exists $\delta>0$ such that the following hold.
		\begin{enumerate}[(1).]
			\item For $\alpha \in [0, \delta)$, $\widetilde{\Theta}_m(\mathbf{\zeta}^{m-2t'+1}_{s'}, \mathbf{\zeta}_{t'})$
			is the unique digraphs with the maximal $\alpha$-spectral radius among $\widetilde{\boldsymbol{\Theta}}_k(m)$.
			\item For $\alpha \in (1-\delta, 1)$, $\widetilde{\Theta}_m(\mathbf{\zeta}_{k-1}^{m-1})$
			is the unique digraph with the maximal  $\alpha$-spectral radius among $\widetilde{\boldsymbol{\Theta}}_k(m)$.
		\end{enumerate}
		
	\end{thm}
	\begin{proof} 
		Take $\mathbf{K}_1 \in \widetilde{S}_{m_1}^{s}, \mathbf{K}'_1 \in \widetilde{S}_{m'_1}^{s'}$ and $\mathbf{K}_2\in \widetilde{S}_{m_2}^{t}, \mathbf{K}'_2\in \widetilde{S}_{m'_2}^{t'}$.
		By Corollary \ref{cors1}, we have
		\begin{align*}
		&\phi_{\alpha}(\widetilde{\Theta}(\mathbf{K}'_1,\mathbf{K}'_2))-\phi_{\alpha}(\widetilde{\Theta}(\mathbf{K}_1,\mathbf{K}_2)) \\
		= & (x-\alpha)^{n-2}\alpha^2(s't'-st)
		+(x-\alpha)^{n}(F(\mathbf{K}_1,\mathbf{K}_2)-F(\mathbf{K}'_1,\mathbf{K}'_2)).
		\end{align*}
		Since $d=\frac{1-\alpha}{x-\alpha}$, for $x>1$, there exists $\delta>0$ such that 
		$\phi_{\alpha}(\widetilde{\Theta}(\mathbf{K}'_1,\mathbf{K}'_2))-\phi_{\alpha}(\widetilde{\Theta}(\mathbf{K}_1,\mathbf{K}_2))$ and $F(\mathbf{K}_1,\mathbf{K}_2)-F(\mathbf{K}'_1,\mathbf{K}'_2)$ have the same sign when $\alpha \in [0, \delta)$, $\phi_{\alpha}(\widetilde{\Theta}(\mathbf{K}'_1,\mathbf{K}'_2))-\phi_{\alpha}(\widetilde{\Theta}(\mathbf{K}_1,\mathbf{K}_2))$ and $s't'-st$ have the same sign when $\alpha \in (1-\delta, 1)$. 
		
		\noindent (1). When $\alpha=0$ and $s\geq t$, by directly calculation, one can see that $F(\mathbf{\zeta}^{m-2t+1}_{s},\mathbf{\zeta}_{t})>F(\mathbf{\zeta}^{m-2t+3}_{s+1},\mathbf{\zeta}_{t-1})$. Since $\phi_{\alpha}(\widetilde{\Theta}(\mathbf{K}_1,\mathbf{K}_2))$ is continuous  about $\alpha$, there exists $\delta>0$ such that $\phi_{\alpha}(\mathbf{\zeta}^{m-2t+1}_{s},\mathbf{\zeta}_{t})<\phi_{\alpha}(\mathbf{\zeta}^{m-2t+3}_{s+1},\mathbf{\zeta}_{t-1})$ holds for $\alpha \in [0, \delta)$.
		This implies that $\widetilde{\Theta}(\mathbf{\zeta}^{m-2t'+1}_{s'},\mathbf{\zeta}_{t'})$ is the unique digraph with the maximal $\alpha$-spectral radius among $\widetilde{\boldsymbol{\Theta}}_k(m)$ for $\alpha \in [0, \delta)$. 
		
		\noindent (2).  Take $s \geq t \geq 2$. For any  $\mathbf{K}_1\in \widetilde{S}_{m_1}^{k-1},\mathbf{K}_2\in \widetilde{S}_{m_2}^{1}$ and $\mathbf{K}'_1\in \widetilde{S}_{m'_1}^{s},\mathbf{K}'_2\in \widetilde{S}_{m'_2}^{t}$ with $m_1+m_2=m'_1+m'_2$, there always exists some $\delta>0$ such that 
		$\phi_{\alpha}(\widetilde{\Theta}(\mathbf{K}'_1,\mathbf{K}'_2))>\phi_{\alpha}(\widetilde{\Theta}(\mathbf{K}_1,\mathbf{K}_2))$ for any $x>1$ and $\alpha \in (1-\delta,1)$. Hence,  $\rho_\alpha(\widetilde{\Theta}(\mathbf{K}_1,\mathbf{K}_2))>\rho_\alpha(\widetilde{\Theta}(\mathbf{K}'_1,\mathbf{K}'_2))$. By $(1)$ of Theorem \ref{thm2}, we have $\widetilde{\Theta}_m(\mathbf{\zeta}_{k-1}^{m-1})$
		is the unique digraph with the maximal  $\alpha$-spectral radius among $\widetilde{\boldsymbol{\Theta}}_k(m)$. Therefore the proof is finished.
	\end{proof}
	
	Note that   there also exists some $\delta>0$ such that  $\widetilde{\Theta}^1_m(\mathbf{\zeta}_{k-1})$
	is the unique digraph with the minimal $\alpha$-spectral radius among $\widetilde{\boldsymbol{\Theta}}_k(m)$ for $\alpha \in [0, \delta)$ by $(1)$ of Theorem \ref{thm2}.

	In \cite{MR2954483} and \cite{MR3777498},  Guo et al. and Li et al. showed that among $\mathcal{R}_{m}^k$ and $\widetilde{\boldsymbol{\Theta}}_k(m)$,  the digraph $\Rose(\mathbf{\gamma}_1)$ is the
	unique digraph with  the maximal adjacency spectral radius and signless Laplacian spectral radius. And the digraph
	$\widetilde{\Theta}_m(\mathbf{\zeta}_{k-1})$ is the unique digraph with  the minimal adjacency spectral radius and signless Laplacian spectral radius among $\mathcal{R}_{m}^k \cup \widetilde{\boldsymbol{\Theta}}_k^1(m)$. In the following,  we will show that the above conclusions for $\alpha$-spectral radius also hold.

	\begin{thm}\label{thm3}\label{thml3}
		Let $G$ be a simple digraph on $n$ vertices, $u,v,w$ be distinct vertices of $V(G)$, $(u,v)\in E(G)$ and $\boldsymbol{x}=(x_1,x_2,\dots,x_n)^T$ be the $\alpha$-Perron vector of $G$, where $x_i$ corresponds to the vertex $i$. Let $G'=G-E_u^{-}+E_v^{-}$ with $E_i^{-}=\{(w,i)|w\in N_{u}^- \}$. If $x_v\ge x_u$, then $\rho_{\alpha}(G')\ge \rho_{\alpha}(G)$.
	\end{thm}
	\begin{proof}
		
		First, we show $(A_{\alpha}(G')\boldsymbol{x})_w\ge (A_{\alpha}(G)\boldsymbol{x})_w$ for any $w\in V(G')=V(G)$.
		By the eigenequations of $A_\alpha(G)$ for $\rho_\alpha(G)$ and the relationship between $G$ and $G'$, the following conclusions hold:\\
		$(A_{\alpha}(G')\boldsymbol{x})_w=(A_{\alpha}(G)\boldsymbol{x})_w=\rho_{\alpha}(G)x_w$ for $w\notin N_u^{-}$  and \\
		$(A_{\alpha}(G')\boldsymbol{x})_w-(A_{\alpha}(G)\boldsymbol{x})_w=\rho_{\alpha}(G)x_w=(1-\alpha)(x_v-x_u)\ge 0$ for $w\in N_u^{-}$.
		
		Thus, $\rho_{\alpha}(G')\boldsymbol{x}=A_{\alpha}(G')\boldsymbol{x}\ge A_{\alpha}(G)\boldsymbol{x}=\rho_{\alpha}(G)\boldsymbol{x}$. From $(iv)$ of Lemma \ref{lemc2}, $\rho_{\alpha}(G')\ge \rho_{\alpha}(G)$ holds.
	\end{proof}

	For digraphs, the following  fact is obvious.
	\begin{fact}\label{cory1}
		The set of $\alpha$-eigenvalues of a digraph is the union (counting multiplicity) of the sets of $\alpha$-eigenvalues of its strong components.
	\end{fact}
	Suppose that $\mathbf{K} \in \widetilde{S}^t_{m}$ and $u_1,v_1$ are two vertices of digraph $G$ with $d_{u_1}^{+}(G)=d_{v_1}^{-}(G)=0$. Let $G_{u\to v}(\mathbf{K})$ be digraph obtained from $G$ and $P_{u_{2}\to v_{2}}(\mathbf{K})$ by identifying/merging  vertices $u_1,u_2$ and  $v_1,v_2$ as new vertices $u$ and $v$, respectively.
	
	\begin{thm}\label{thms3}
		Suppose that $H=G_{u\to v}(\mathbf{K})$ is strongly connected digraph with $V(H)=\{1,2,\ldots,n\}$.
		Let $E_i^{-}=\{(w,i)|w\in {u}^-(H)\}$ and  $E_i^{+}=\{(i,w)|w\in N_{u}^+(H) \}$ for $i \in \{u,v\}$ and  $H'=H-E_u^{-}-E_u^{+}+E_v^{-}+E_v^{+}$. Suppose that  $\boldsymbol{x}=(x_1, x_2, \dots, x_n)^T$ is  the $\alpha$-Perron vector of $H$,  where $x_i$ corresponds to the vertex $i$.  If $x_v\ge x_u$,  then $\rho_{\alpha}(H')>  \rho_{\alpha}(H)$.

	\end{thm}
	\begin{proof}
		Take  $H_1=H-E_u^{-}+E_v^{-}$ and  $G_1$ is digraph obtained from $G$ by identifying the vertices $u_1, v_1$ as a new vertex. Since  the strongly connected components of $H_1$ are $G_1$ and some isolated vertices. $\rho_{\alpha}(H_1)=\rho_{\alpha}(G_1)$ follows from Fact \ref{cory1}. If $x_v\ge x_u$,  then by Theorem \ref{thm3}, we have $\rho_{\alpha}(H_1)\ge  \rho_{\alpha}(H)$.   Since $G_1$ is a proper subdigraph of $H'$, by $(ii)$ of Lemma \ref{lemc2}, we have $\rho_{\alpha}(H')>\rho_{\alpha}(G_1)$. Hence, we have that $\rho_{\alpha}(H')>\rho_{\alpha}(G_1)=\rho_{\alpha}(H_1)\ge  \rho_{\alpha}(H)$, that is $\rho_{\alpha}(H')>  \rho_{\alpha}(H)$.
	\end{proof}

	\begin{lem}\label{lem12}Suppose that $G \in \widetilde{\boldsymbol{\Theta}}_k(m)$ and $H \in \mathcal{R}_{m}^k$. Then
		\begin{enumerate}[(1).]
			\item there exists some digraph $G' \in \mathcal{R}_{m}^k$ such that
			$\rho_{\alpha}(G)<\rho_{\alpha}(G')$;
			\item  there exists some digraph $H' \in \widetilde{\boldsymbol{\Theta}}_k^{1}(m)$ such that
			$\rho_{\alpha}(H')<\rho_{\alpha}(H)$.
		\end{enumerate}
	\end{lem}
	\begin{proof}
		$(1)$ Wlog, assume that $G=\widetilde{\Theta}(\mathbf{K}_1,\mathbf{K}_2)$ with $\mathbf{K}_1=(k_1,k_2,\dots,k_s) \in \widetilde{S}_{m_1}^s$, $\mathbf{K}_2=(l_1,l_2,\dots,l_t)\in \widetilde{S}_{m_2}^{t}$ and  $u,v \in V(G)$ such that
		$d^{+}_{G}(u)=d^{-}_{G}(v)=s, d^{+}_{G}(v)=d^{-}_{G}(u)=t$.
		We can assume w.l.o.g. that $\boldsymbol{x}$ be the $\alpha$-Perron vector of $G$ with $x_v\ge x_u$.
		
		For $i  \in \{u,v\}$, take $E_i^{-}=\{(w,i)|w\in N_{u}^-\}$ and $E_i^{+}=\{(i,w)|w\in N_{u}^+ \}$.
		Let $G'=G-E_u^{-}-E_u^{+}+E_v^{-}+E_v^{+}$, Then $G'=\Rose(\mathbf{K})\cup \{u\}$ with $\mathbf{K}=(\mathbf{K}_1,\mathbf{K}_2)\in \widetilde{S}_{m_1+m_2}^{s+t}$.
		By Theorem \ref{thms3}, we have that $\rho_{\alpha}(G')>\rho_{\alpha}(G)$. Hence, we have $\rho_{\alpha}(\Rose(\mathbf{K}))=\rho_{\alpha}(G')>\rho_{\alpha}(G)$.
		
		$(2)$ Similar to proof of $(1)$, by Theorem \ref{thml3}, the conclusion of $(2)$ follows.
	\end{proof}
	
	From Lemma \ref{lem12},  we know that the digraph with the maximal $\alpha$-spectral radius among $\mathcal{R}_{m}^k \cup \widetilde{\boldsymbol{\Theta}}_k(m)$ must be in $\mathcal{R}_{m}^k$,  and the digraph with  the minimal $\alpha$-spectral radius among $\mathcal{R}_{m}^k \cup \widetilde{\boldsymbol{\Theta}}_k^1(m)$ must be in $\widetilde{\boldsymbol{\Theta}}_k^1(m)$. By Theorem \ref{thm4}, $(1)$ of Theorem \ref{thm2} and  Lemma \ref{lem12},  the next theorem follows immediately.
	\begin{thm}\label{thm7}
		Take $\gamma_1=\mathbf{x}^*(S_m^{k})$, among $\mathcal{R}_{m}^k \cup \widetilde{\boldsymbol{\Theta}}_k(m)$,  the digraph $\Rose(\gamma_1)$ is the unique digraph with  the maximal $\alpha$-spectral radius,
		and  the digraph $\widetilde{\Theta}_m(\mathbf{\zeta}_{k-1})$
		is the unique digraph with  the minimal $\alpha$-spectral radius among $\mathcal{R}^{k}_m \cup \widetilde{\boldsymbol{\Theta}}_{k}^{1}(m)$.
	\end{thm}

	By Theorem \ref{thm7}, we have the next corollary, 	which is Theorem $3.3$ in \cite{MR3872980}.
	
	\begin{cor}
		$\infty(2, m-2)$ and $\widetilde{\Theta}(2,1; m-3)$ are the unique digraphs with the maximal and minimal $\alpha$-spectral radius, among all the strongly connected bicyclic digraphs with size $m$, respectively.
	\end{cor}
	
	\section{The first four minimal \texorpdfstring{$\alpha$}{α}-spectral radius of strongly connected digraphs}
	In \cite{MR3230435},  Hong and You determined the first four minimal signless Laplacian spectral radii and adjacency spectral radii among all strongly connected digraphs on $n$ vertices.
	In this section,  we will show that the above conclusions for $\alpha$-spectral radius also hold.
	
	The girth of $G$ is the length of the shortest directed cycle of $G$. Let $\mathcal{G}_{n, g}$  denote the set of strongly connected digraph on $n$ vertices with girth $g\ge 2$. If $g=n$,  then $\mathcal{G}_{n, n}=\overrightarrow{C_n}$. For $2\le g\le n-1$, take $C_{n, g}$ be a digraph  (see (a) of Fig.\ref{fig:twofig}) obtained by adding a directed edge $(u_g, u_1)$ on the directed cycle $\overrightarrow{C_n}=u_1u_2\dots u_n$. Clearly,  $C_{n, g}\in \mathcal{G}_{n, g}$ and $C_{n, g}\cong \widetilde{\Theta}(n+1-g, 1; g-1)$.

	The following lemmas will be used in the proof of our main results.
	
	\begin{lem}\cite{Xi2018Merging}\label{lem14}
		\begin{enumerate}[(i).]
			\item Let $D(\neq \overrightarrow{C_n})$ be a strongly connected digraph with $V(D)=\{u_1, u_2, \dots, u_n\}$,  $\overrightarrow{P}=u_1u_2\dots u_l(l\ge 3)$ be a directed path of $D$ with $d_{u_i}^+=1(i=2, 3, \dots, l-1)$.  Then we have $x_2<x_3<\dots<x_{l-1}<x_l$.
			\item 	For $2\le g\le n-1$ and $G\in \mathcal{G}_{n, g}$, $\rho_{\alpha}(G)\ge \rho_{\alpha}(C_{n, g})>1$ holds, and equality if and only if $G\cong C_{n, g}$.
		\end{enumerate}
	\end{lem}
	
	\begin{lem}\label{lem16}
		Let $n\ge 4$, then $\rho_{\alpha}(C_{n, g+1})<\rho_{\alpha}(C_{n, g})$.
	\end{lem}
	\begin{proof}
		It is clear $\rho_{\alpha}(\overrightarrow{C_n})=1$. Now we only need to show that $\rho_{\alpha}(C_{n, g+1})<\rho_{\alpha}(C_{n, g})$ for $2\le g\le n-2$.
		
		Let $\boldsymbol{x}=(x_1, x_2, \dots, x_n)^T$ be the $\alpha$  Perron vector  of $C_{n, g+1}$,  while $x_i$ corresponds to the vertex $u_i$. Clearly,  $C_{n, g}\cong C_{n, g+1}-(u_{g+1}, u_1)+(u_{g+1}, u_2)$. Since $\overrightarrow{P}=u_n u_1 u_2\dots u_{g+1}$ is directed path of $C_{n, g+1}$ with $d_{u_i}^+=1$ where $i\in\{1, 2, \dots, n\}$,  then by $(i)$ of Lemma \ref{lem14},  we have $x_n<x_1<x_2<\dots<x_{g+1}$. Then we have $\rho_{\alpha}(C_{n, g+1})<\rho_{\alpha}(C_{n, g})$ by Theorem \ref{thml3}.
	\end{proof}

	\begin{figure}
		\begin{center}
		  \subfigure[$C(n,g)$]{\includegraphics[page=5,  trim=0 0 0 0,  clip, width=0.4\textwidth]{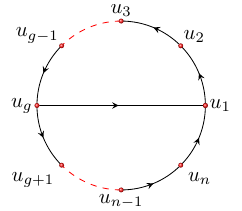}} \hfill
		  \subfigure[$\widehat{\Theta}(n)$]{\includegraphics[page=4,  trim=0 0 0 0,  clip, width=0.4\textwidth]{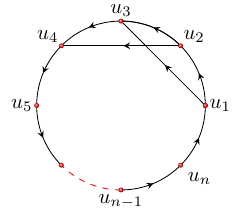}}
		\end{center}
		\vspace*{-12pt}
		\caption{Two digraphs}
		\label{fig:twofig}
	  \end{figure}

	Take $\widehat{\Theta}(n)=\widetilde{\Theta}(2, 2; n-3)\cup \{(u_2, u_3)\}$ as depicted in (b) of Fig. \ref{fig:twofig}.
	
	\begin{thm}
		Among all strongly connected digraphs on $n(n\ge 4)$ vertices, $\overrightarrow{C_n}$, $\widetilde{\Theta}(2, 1; n-2)$ and  $\widetilde{\Theta}(2, 2; n-3)$ are the digraphs with  the first three minimal $\alpha$-spectral radius,  respectively. And  $\widetilde{\Theta}(3, 1; n-3)$ achieves the fourth minimal $\alpha$-spectral radius for $0\le \alpha \le \frac{1}{2}$.
	\end{thm}
	
	\begin{proof}
		Noting that $\mathcal{G}_{n, n}=\{\overrightarrow{C_n}\}$, $\mathcal{G}_{n, n-1}=\{\widetilde{\Theta}(2, 1; n-2), \widetilde{\Theta}(2, 2; n-3), \widehat{\Theta}(n)\}$, $C_{n, n-1}\cong \widetilde{\Theta}(2, 1; n-2)$ and $C_{n, n-2}\cong\widetilde{\Theta}(3, 1; n-3)\in \mathcal{G}_{n, n-2}$.
		By $(ii)$ of Lemma \ref{lem14} and Lemma \ref{lem16}, it is clear that $$\rho_{\alpha}(\overrightarrow{C_n})<\rho_{\alpha}(\widetilde{\Theta}(2, 1; n-2)) <\min\{\rho_{\alpha}(\widetilde{\Theta}(2, 2; n-3)), \rho_{\alpha}(\widehat{\Theta}(n)),\rho_{\alpha}(\widetilde{\Theta}(3, 1; n-3))\}.$$ 
		
		It is sufficient to show that
		$\rho_{\alpha}(\widetilde{\Theta}(2, 2; n-3))<\rho_{\alpha}(\widetilde{\Theta}(3, 1; n-3))<\rho_{\alpha}(\widehat{\Theta}(n))$. 
		
		By $(ii)$ of Lemma \ref{lemc2} and $\widehat{\Theta}(n)=\widetilde{\Theta}(2, 2; n-3)\cup \{(u_2, u_3)\}$,  then we have $\rho_{\alpha}(\widetilde{\Theta}(2, 2; n-3))<\rho_{\alpha}(\widehat{\Theta}(n)).$ 
		
		Take $\mathbf{K}_1=(2,2),\ \mathbf{K}'_1 =(3,1),\ \mathbf{K}_2=(n-3),\  \mathbf{K}'_2=(n-3)$, and $\mathbf{K}_1 \prec \mathbf{K}'_1$,\ $\mathbf{K}_2 =\mathbf{K}'_2 $. By Lemma \ref{lems1}, we have $\rho_{\alpha}(\widetilde{\Theta}(2, 2; n-3))<\rho_{\alpha}(\widetilde{\Theta}(3, 1; n-3)).$
		
		By direct calculation, we have
		\begin{equation}\label{chap:theta1}
		\begin{aligned}
		\phi_{\alpha}(\widehat{\Theta}(n))= & (x-2\alpha)^2(x-\alpha)^{n-2}-(x-2\alpha)(1-\alpha)^{n- 1} \\
		& -(x-\alpha)(1-\alpha)^{n-1}-(1-\alpha)^n.
		\end{aligned}
		\end{equation}
		
		$\Delta (x)  =(\phi_{\alpha}(\widetilde{\Theta}(3, 1; n-3)-\phi_{\alpha}(\widehat{\Theta}(n))(1-\alpha)^{3-n}$. By Formulae \eqref{theta:chp} and  \eqref{chap:theta1},   we have
		\begin{align*}
		\Delta(x) & =\alpha(x-2\alpha)(x-\alpha)d^{3-n}-[(x-1)^2+(2\alpha-1)(1-\alpha)](1-\alpha) \\
		& \ge \alpha(x-2\alpha)(x-\alpha)-[(x-1)^2+(2\alpha-1)(1-\alpha)](1-\alpha)     \\
		\,        & =(2\alpha-1)x^2+(2-3\alpha^2-2\alpha)x+5\alpha^2-3\alpha=\Delta_2(x) .
		\end{align*}
		When $0\le \alpha<\frac{1}{2}$,  $2\alpha-1<0$,  then $\Delta_2(1)=2\alpha^2-3\alpha+1>0$ and $\Delta_2(2)=\alpha(1-\alpha)>0$, for any $1\le x\le 2$, we have $\Delta_2(x)>0$;
		when $\alpha=\frac{1}{2}$,  $\Delta_2(x)=\frac{1}{4}(x-1)>0$,  for any $max\{1, 2\alpha\}<x<2$. Then we have $\Delta(x)>0$, when $0\le \alpha\le  \frac{1}{2}$.
		So,  we have that $\rho_{\alpha}(\widetilde{\Theta}(3, 1; n-3))<\rho_{\alpha}(\widehat{\Theta}(n)).$
		
	\end{proof}
	
	We finish this section with the following conjecture.
	\begin{con}
		When $\frac{1}{2}< \alpha<1$, then
		$$\rho_{\alpha}(\widetilde{\Theta}(3, 1; n-3))<\rho_{\alpha}(\widehat{\Theta}(n)).$$
	\end{con}

\end{document}